\documentclass[10pt]{article}
\usepackage{latexsym}
\usepackage{amssymb}
\usepackage{amsthm}
\usepackage{amsmath}
\usepackage{mathrsfs}
\usepackage{cite}
\usepackage{textcomp}

\newtheorem{theorem}{Theorem}[section]
\newtheorem{lemma}[theorem]{Lemma}
\newtheorem{corollary}[theorem]{Corollary}

\numberwithin{equation}{section}

\newcommand{\R}{{\mathscr R}}

\newcommand{\nb}{\nonumber}

\begin{document}


\begin{center}
{\LARGE Reverse order laws for generalized inverses  of products of two or three matrices with applications}
\end{center}

\begin{center}
{\Large Yongge Tian}
\end{center}

\begin{center}
{\footnotesize \em
Shanghai Business School, Shanghai, China \& China Central University of Finance and Economics, Beijing, China}
\end{center}

\renewcommand{\thefootnote}{\fnsymbol{footnote}}
\footnotetext{{\it E-mail:} yongge.tian@gmail.com}

{\small \noindent{\bf  Abstract.}
One of the fundamental research problems in the theory of generalized inverses of matrices is to establish reverse order laws for generalized inverses of
matrix products. Under the assumption that $A$, $B$, and $C$ are three nonsingular matrices of the same size, the products $AB$ and $ABC$ are nonsingular as well, and the inverses of $AB$ and $ABC$ admit the reverse order laws $(AB)^{-1} = B^{-1}A^{-1}$ and $(ABC)^{-1} = C^{-1}B^{-1}A^{-1}$, respectively. If some or all of $A$, $B$, and $C$ are singular, two extensions of the above reverse order laws to generalized inverses can be written as $(AB)^{(i,\ldots,j)} = B^{(i_2,\ldots,j_2)}A^{(i_1,\ldots,j_1)}$ and $(ABC)^{(i,\ldots,j)} = C^{(i_3,\ldots,j_3)}B^{(i_2,\ldots,j_2)}A^{(i_1,\ldots,j_1)}$, or other mixed reverse order laws. These equalities do not necessarily hold for different choices of generalized inverses of the matrices. Thus it is a tremendous work to classify and derive necessary and sufficient conditions for the reverse order law to hold because there are all 15 types of $\{i,\ldots, j\}$-generalized inverse for a given matrix according to combinatoric choices of the four Penrose equations. In this paper, we first establish several decades of mixed reverse order laws for $\{1\}$- and $\{1,2\}$-generalized inverses of $AB$ and $ABC$. We then give a classified investigation to a special family of reverse order laws $(ABC)^{(i,\ldots,j)} = C^{-1}B^{(k,\ldots,l)}A^{-1}$ for the eight commonly-used types of generalized inverses using definitions, formulas for ranges and ranks of matrices, as well as conventional operations of matrices. Furthermore, the special cases $(ABA^{-1})^{(i,\ldots,j)} = AB^{(k,\ldots,l)}A^{-1}$ are addressed and some applications are presented.\\


\medskip

}

\medskip

\section{Introduction}

Throughout this article, let ${\mathbb C}^{m\times n}$ denote the set of all $m\times n$ complex matrices. $A^{*}$,  $r(A)$, and ${\mathscr R}(A)$ be the conjugate transpose, the rank, and the range (column space) of a matrix $A\in {\mathbb C}^{m\times n}$, respectively; $I_m$ be the identity matrix of order $m$; and $[\, A, \, B\,]$ be a row block matrix consisting of $A$ and $B$. A matrix $A \in {\mathbb C}^{m \times m}$ is said to be EP (or range Hermitian) if ${\mathscr R}(A^{*}) ={\mathscr R}(A)$
holds. We next introduce the definition and notation of generalized inverses
of a matrix. The Moore--Penrose inverse of $A \in {\mathbb C}^{m \times n}$, denoted by $A^{\dag}$, is the unique
matrix $X \in {\mathbb C}^{n \times m}$ satisfying the four Penrose equations
\begin{align}
{\rm (i)}  \  AXA = A,   \ \   {\rm (ii)}   \ XAX=X, \ \ {\rm (iii)}  \  (AX)^{*} = AX,  \ \  {\rm (iv)}   \  (XA)^{*} = XA.
 \label{11}
\end{align}
A matrix $X$ is called an $\{i,\ldots, j\}$-generalized inverse of $A$, denoted by $A^{(i,\ldots, j)}$, if it satisfies the $i$th,$\ldots,j$th equations in \eqref{11}. The collection of all $\{i,\ldots, j\}$-generalized inverses of $A$ is denoted by
$\{(A^{(i,\ldots,j)}\}$.  There are all 15 types of $\{i,\ldots, j\}$-generalized inverses of $A$, but
$A^{\dag}$,  $A^{(1,3,4)}$, $A^{(1,2,4)}$,  $A^{(1,2,3)}$, $A^{(1,4)}$, $A^{(1,3)}$,  $A^{(1,2)}$, and  $A^{(1)}$
are usually called the eight commonly-used types of generalized inverses of $A$ in the literature.

One of the fundamental objects of study in the theory of generalized inverses is to establish matrix equalities that involve
generalized inverses, including various types of reverse order law for generalized inverses of matrix products. For three nonsingular matrices
$A$, $B$, and $C$ of the same size, the products $AB$ and $ABC$ are nonsingular as well, and the reverse order laws $(AB)^{-1} = B^{-1}A^{-1}$ and
$(ABC)^{-1} = C^{-1}B^{-1}A^{-1}$, as well as the cancellation law $C(ABC)^{-1}A = B^{-1}$ always hold. These identities can be used to simplify matrix expressions that involve inverse operations of products of  nonsingular matrices. If some or all of $A$, $B$, and $C$ are singular, generalized inverses of $AB$ and $ABC$ can be written as certain expressions composed by $A$, $B$, and $C$ and their generalized inverses. In particular, two families of extensions of the above identities to generalized inverses of $AB$ and $ABC$ can be written as
\begin{align}
(AB)^{(i,\ldots,j)} = B^{(i_2,\ldots,j_2)}A^{(i_1,\ldots,j_1)}, \ \  (ABC)^{(i,\ldots,j)} = C^{(i_3,\ldots,j_3)}B^{(i_2,\ldots,j_2)}A^{(i_1,\ldots,j_1)},
\label{12}
\end{align}
which have been extensively approached in the theory of generalized inverses since 1960s. In addition to \eqref{12}, generalized inverses of the matrix products $AB$ and $ABC$ can be written as various mixed ROLs, such as,
\begin{align}
& (AB)^{(i,\ldots,j)} = B^{(i_2,\ldots,j_2)}XA^{(i_1,\ldots,j_1)},  \ \ \ \ \ \ \ \ \ \ \ \ \ \ \ \ \ \ \ (AB)^{(i,\ldots,j)} = B^{(i_2,\ldots,j_2)}A^{(i_1,\ldots,j_1)} + Y,
\label{gg13}
\\
& (ABC)^{(i,\ldots,j)} = C^{(i_3,\ldots,j_3)}YB^{(i_2,\ldots,j_2)}XA^{(i_1,\ldots,j_1)},  \ \
(ABC)^{(i,\ldots,j)} =  C^{(i_3,\ldots,j_3)}B^{(i_2,\ldots,j_2)}A^{(i_1,\ldots,j_1)} + Z,
\label{gg14}
\end{align}
etc., for certain matrices $X$, $Y$, and $Z$ composed by $A$, $B$, $C$, and their generalized inverses.
Since generalized inverses of a matrix are not necessarily unique,  \eqref{12} can also be described by the following matrix set equalities
\begin{align}
& \{(AB)^{(i,\ldots,j)} \} = \{B^{(i_2,\ldots,j_2)}A^{(i_1,\ldots,j_1)}\},  \ \ \ \  \ \ \  \{(ABC)^{(i,\ldots,j)} \} = \{C^{(i_3,\ldots,j_3)}B^{(i_2,\ldots,j_2)}A^{(i_1,\ldots,j_1)}\},
\label{12a}
\\
& \{(AB)^{(i,\ldots,j)}\} = \{B^{(i_2,\ldots,j_2)}XA^{(i_1,\ldots,j_1)}\},  \ \ \ \ \{(AB)^{(i,\ldots,j)}\}= \{B^{(i_2,\ldots,j_2)}A^{(i_1,\ldots,j_1)} + Y\},
\label{xx13}
\\
& \{(ABC)^{(i,\ldots,j)}\} = \{C^{(i_3,\ldots,j_3)}YB^{(i_2,\ldots,j_2)}XA^{(i_1,\ldots,j_1)}\},
\\
& \{(ABC)^{(i,\ldots,j)} \} =  \{C^{(i_3,\ldots,j_3)}B^{(i_2,\ldots,j_2)}A^{(i_1,\ldots,j_1)} + Z \}.
\label{ss14}
\end{align}
These equalities do not necessarily hold for different choices of generalized inverses of the matrices. Thus we wish to find identifying conditions for
 \eqref{12} and \eqref{12a} to hold under various assumptions. This is really a tremendous work because there are all 15 types of $\{i,\ldots, j\}$-generalized inverse for a given matrix according to combinatoric choices of the four Penrose equations. Although approach on ROLs was started in 1960s and has been one of the attractive and fruitful research fields in matrix algebra and operator theory, only a small part of these ROLs were considered; for instance,
 $(AB)^{(1)} = B^{(1)}A^{(1)}$ and $(AB)^{\dag} = B^{\dag}A^{\dag}$ were approached in \cite{Arg,Gre,Erd,SS1,SS2,Sib,Tian:2007,Wer,WHG} among others; $(ABC)^{\dag} = C^{\dag}B^{\dag}A^{\dag}$ was considered in \cite{DW:2009,DD:2014,Har:1986,LT:2009,Tian:1994,TL:2007}, while mixed ROLs for Moore--Penrose inverses of  $AB$ and $ABC$ were formulated and approached in \cite{DD,DDM,GW,Iz,NL,Tian:2004b,Tian:2005,Tian:2007,Tre} among others. In spite of many efforts, most of \eqref{12} remain unresolved. In the first part of this article, we formulate a variety of mixed ROLs for generalized inverses of $AB$ and $ABC$. We then consider a special case of $ABC$ with $A$ and $C$ being nonsingular. In this case, we can divide \eqref{12a} into the following four matrix set relations
\begin{align}
& \{(ABC)^{(i,\ldots,j)}\} \cap \{C^{-1}B^{(k,\ldots,l)}A^{-1}\} \neq \emptyset \ \ \ \ \ (8^2 = 64 \ {\rm situations}),
\label{14}
\\
&\{(ABC)^{(i,\ldots,j)}\} \supseteq  \{C^{-1}B^{(k,\ldots,l)}A^{-1}\}   \ \ \ \ \ \ \ \ \ \ (8^2 = 64 \ {\rm situations}),
\label{15}
\\
& \{(ABC)^{(i,\ldots,j)}\} \subseteq  \{C^{-1}B^{(k,\ldots,l)}A^{-1}\}  \ \ \ \ \ \ \ \ \ \ (8^2 = 64 \ {\rm situations}),
\label{16}
\\
& \{(ABC)^{(i,\ldots,j)}\} = \{C^{-1}B^{(k,\ldots,l)}A^{-1}\}  \ \ \  \ \ \ \ \ \ \ (8^2 = 64 \ {\rm situations})
\label{17}
\end{align}
for the eight common-used generalized inverses of $ABC$ and $B$. The usual procedure of establishing and describing the matrix equalities in  \eqref{14}--\eqref{17} is to employ definitions of generalized inverses,  range and rank formulas for matrices, and matrix equations. The aim of  this paper is to derive necessary and sufficient conditions for \eqref{14}--\eqref{17} to hold respectively using these usual techniques.

\section[2]{Preliminaries}

In this section, we enumerate some basic facts concerning generalized inverses of matrices that we shall use in the sequel.

\begin{lemma} [\cite{BG,CM,RM}]\label{T11}
Let $A \in {\mathbb C}^{m \times n}.$ Then the following results hold$.$
\begin{enumerate}
\item[{\rm (a)}] $A^{\dag}$ satisfies the following two equalities
\begin{align}
& (A^{\dag})^{*} = (A^{*})^{\dag}, \ \ \ \  (A^{\dag})^{\dag} = A,
 \label{18}
\\
& (A^{*})^{\dag}A^{*} = (AA^{\dag})^{*} = AA^{\dag}, \ \ A^{*}(A^{*})^{\dag} = (A^{\dag}A)^{*} = A^{\dag}A,
 \label{19}
\\
& {\mathscr R}(A)=  {\mathscr R}(AA^{*}) =  {\mathscr R}(AA^{*}A) =
{\mathscr R}(AA^{\dag}) = {\mathscr R}[(A^{\dag})^{*}],
\label{110}
\\
& {\mathscr R}(A^{*})=  {\mathscr R}(A^{*}A) =  {\mathscr R}(A^{*}AA^{*})
={\mathscr R}(A^{\dag}) = {\mathscr R}(A^{\dag}A),
\label{111}
\\
& r(A) = r(A^{*}) = r(A^{\dag}) = r(AA^{*}) = r(A^{*}A) =r(AA^{\dag}) = r(A^{\dag}A).
\label{112}
\end{align}

\item[{\rm (b)}] The general expressions of the seven commonly-used types of generalized inverses $A^{(1,3,4)}$, $A^{(1,2,4)}$, $A^{(1,2,3)}$,  $A^{(1,4)}$,  $A^{(1,3)}$, $A^{(1,2)}$, and $A^{(1)}$ of $A$ can be written in the following 7 matrix-valued functions
\begin{align}
& A^{(1)} = A^{\dag} + F_{A}U_1 + U_2E_{A},
 \label{113}
\\
& A^{(1,2)} = (\, A^{\dag} + F_{A}U_1\,)A(\, A^{\dag}  + U_2E_{A} \,),
\label{114}
\\
& A^{(1,3)} = A^{\dag} + F_{A}U,
\label{115}
\\
& A^{(1,4)} = A^{\dag} + UE_{A},
\label{116}
\\
& A^{(1,2,3)} = A^{\dag} + F_{A}UAA^{\dag},
\label{117}
\\
& A^{(1,2,4)} = A^{\dag} + A^{\dag}AUE_{A},
\label{118}
\\
& A^{(1,3,4)} = A^{\dag} + F_{A}UE_{A},
\label{119}
\end{align}
where $U,\, U_1, \, U_2 \in {\mathbb C}^{n \times m}$ are arbitrary$.$ 

\item[{\rm (c)}] The following matrix equalities hold
\begin{align}
& AA^{(1)} = AA^{(1,2)} =  AA^{(1,4)} = AA^{(1,2,4)}  =AA^{\dag} +  AUE_{A},
\label{120}
\\
& AA^{(1,3)} = AA^{(1,2,3)}  = AA^{(1,3,4)} = AA^{\dag},
\label{121}
\\
&A^{(1)}A =  A^{(1,2)}A = A^{(1,3)}A  = A^{(1,2,3)}A  = A^{\dag}A + F_{A}UA,
\label{122}
\\
& A^{(1,4)}A =  A^{(1,2,4)}A = A^{(1,3,4)}A  = A^{\dag}A,
\label{123}
\end{align}%
where $U\in {\mathbb C}^{n \times m}$ is arbitrary$.$

\item[{\rm (d)}] The following set inclusions hold
\begin{align}
&  A^{\dag} \in \{ A^{(1,2,3)} \} \subseteq  \{ A^{(1,2)} \}  \subseteq  \{ A^{(1)} \},
\label{s6}
\\
& A^{\dag} \in \{ A^{(1,2,3)} \} \subseteq  \{ A^{(1,3)} \}  \subseteq  \{ A^{(1)} \},
\label{s5}
\\
& A^{\dag} \in \{ A^{(1,2,4)} \} \subseteq  \{ A^{(1,2)} \}  \subseteq  \{ A^{(1)} \},
\label{s4}
\\
& A^{\dag} \in \{ A^{(1,2,4)} \} \subseteq  \{ A^{(1,4)} \}  \subseteq  \{ A^{(1)} \},
\label{s3}
\\
& A^{\dag} \in \{ A^{(1,3,4)} \} \subseteq  \{ A^{(1,3)} \}  \subseteq  \{ A^{(1)} \}.
\label{s2}
\\
& A^{\dag} \in \{ A^{(1,3,4)} \} \subseteq  \{ A^{(1,4)} \}  \subseteq  \{ A^{(1)} \},
\label{s1}
\end{align}

\item[{\rm (e)}]  The following matrix set equalities hold
\begin{align}
&\{(A^{(1,3,4)})^{*}\} = \{(A^{*})^{(1,3,4)}\}, \ \ \  \{(A^{(1,2,4)})^{*}\} = \{(A^{*})^{(1,2,3)}\},
\label{130}
\\
& \{(A^{(1,2,3)})^{*}\} =\{(A^{*})^{(1,2,4)}\}, \ \ \  \{(A^{(1,4)})^{*}\} = \{(A^{*})^{(1,3)}\},
\label{131}
\\
& \{(A^{(1,3)})^{*}\} =\{(A^{*})^{(1,4)}\}, \ \ \ \ \ \  \{(A^{(1,2)})^{*}\} = \{(A^{*})^{(1,2)}\},
\label{132}
\\
&  \{(A^{(1)})^{*}\} = \{(A^{*})^{(1)}\}.
\label{133}
\end{align}

\item[{\rm (f)}]  The following rank equalities
\begin{align}
& r(A^{(1,2,4)}) = r(A^{(1,2,3)}) = r(A^{(1,2)}) =  r(A^{\dag})= r(A)
\label{134}
\end{align}
hold for all $A^{(1,2,4)}),$ $A^{(1,2,3)},$ and $A^{(1,2)}),$  and the following rank equalities hold
\begin{align}
& \max_{A^{(1)}} r(A^{(1)}) = \max_{A^{(1,3)}} r(A^{(1,3)}) = \max_{A^{(1,4)}} r(A^{(1,4)}) =
\max_{A^{(1,3,4)}}\!r(A^{(1,3,4)}) = \min\{\, m, \, n \,\},
\label{135}
\\
& \min_{A^{(1)}} r(A^{(1)}) =  \min_{A^{(1,3)}} r(A^{(1,3)}) = \min_{A^{(1,4)}} r(A^{(1,4)})
 =  \min_{A^{(1,3,4)}}\!r(A^{(1,3,4)}) = r(A).
\label{136}
\end{align}

\item[{\rm (g)}]  The following equivalent facts hold
\begin{align}
&  G \in  \{ A^{(1)} \} \Leftrightarrow   AGA = A,
\label{g1}
\\
& G \in  \{A^{(1,2)} \} \Leftrightarrow  AGA = A  \ and  \ r(G) = r(A),
\label{g2}
\\
& G \in  \{A^{(1,3)} \} \Leftrightarrow  A^{*}AG = A^{*},
\label{g3}
\\
& G \in  \{A^{(1,4)} \} \Leftrightarrow  GAA^{*} = A^{*},
\label{g4}
\\
& G \in  \{A^{(1,2,3)} \} \Leftrightarrow  A^{*}AG = A^{*} \  and \ r(G) = r(A),
\label{g5}
\\
& G \in \{A^{(1,2,4)} \} \Leftrightarrow  GAA^{*} = A^{*}  and \ r(G) = r(A),
\label{g6}
\\
& G \in  \{A^{(1,3,4)} \} \Leftrightarrow  A^{*}AG = A^{*}  \ and \ GAA^{*} = A^{*},
\label{g7}
\\
& G = A^{\dag} \Leftrightarrow  A^{*}AG = A^{*}, \ GAA^{*} = A^{*},  \ and \ r(G) = r(A).
\label{g8}
\end{align}
\end{enumerate}
\end{lemma}

In order to establish and simplify various matrix equalities that involve generalized inverses, we need the
following well-known rank formulas.

\begin{lemma}  [\cite{MS:1974}] \label{Th2}
Let $ A \in {\mathbb C}^{m \times n}, \, B \in {\mathbb C}^{ m \times k},$ and $C \in {\mathbb C}^{l \times n}.$ Then
\begin{align}
r[\, A, \, B \,]  =  r(A) + r(B - AA^{(1)}B), \ \ \ \!\begin{bmatrix}  A  \\ C  \end{bmatrix} = r(A) + r(C - CA^{(1)}A)
\label{r12}
\end{align}
hold for all $A^{(1)}.$
\end{lemma}

\begin{lemma}[\cite{MS:1974}]  \label{L3}
Let $A \in {\mathbb C}^{m \times n}, \, B \in {\mathbb C}^{n \times p},$ and $C \in \mathbb C^{p \times q}.$ Then
\begin{align}
& r(AB)  = r(A) + r(B) - n + r[(I_n - BB^{(1)})(I_p - A^{(1)}A)],
\label{s17}
\\
& r(ABC)  = r(AB)  + r(BC) - r(B) + r[(I_n - (BC)(BC)^{(1)})B(I_p - (AB)^{(1)}(AB))]
\label{s18}
\end{align}
hold for all $A^{(1)},$ $B^{(1)},$ $(AB)^{(1)},$ and $(BC)^{(1)}.$
\end{lemma}

\begin{lemma} [\cite{Tian:2004}] \label{L2}
Let $A \in {\mathbb C}^{m \times n}, \, B \in {\mathbb C}^{ m \times k},$
$C \in {\mathbb C}^{l \times n},$ and $D \in {\mathbb C}^{l \times k}.$
Then
\begin{align}
\max_{A^{(1,2)}} r(\, D - CA^{(1,2)}B \,) & = \min  \left\{ r(A) + r(D), \,
r[\, C, \, D \,], \,
r\!\begin{bmatrix} B \\ D  \end{bmatrix}\!,  \
 r\!\begin{bmatrix}  A  & B \\ C  & D \end{bmatrix}
 -r(A)
 \right\}\!,
\label{r14}
\\
 \min_{A^{(1,2)}} r(\, D - CA^{(1,2)}B \,) & =  r\!\begin{bmatrix} B \\ D
\end{bmatrix} + r[ \, C, \  D \, ] + r(A) + \max \{ \, r_1, \  r_2 \,\},
 \label{r15}
\end{align}
where
\begin{align*}
r_1 & = r\!\begin{bmatrix} A  & B \\ C & D
 \end{bmatrix}\!
 - r\!\begin{bmatrix} A  & 0  & B \\ 0   & C & D \end{bmatrix} - r\!\begin{bmatrix} A  & 0 \\ 0   & B \\ C & D
\end{bmatrix}\!, \ \ \  r_2 = r(D) -  \!r\!\begin{bmatrix} A  & 0 \\ C  & D
 \end{bmatrix}\! -r\!\begin{bmatrix} A  & B \\ 0   & D  \end{bmatrix}\!.
\end{align*}
\end{lemma}

\begin{lemma} [\cite{TJ:2018}] \label{L}
Let $A \in {\mathbb C}^{m \times n},$  $B \in {\mathbb C}^{m\times q},$ and $C \in {\mathbb C}^{p \times q}.$ Then the following two formulas hold
\begin{align}
&\max_{A^{(1,2)}, C^{(1,2)}}\!\!\!\! r(A^{(1,2)}BC^{(1,2)}) =  \min\{\, r(A),  \ r(B),  \ r(C) \, \},
\label{gg18}
\\
&\min_{A^{(1,2)}, C^{(1,2)}}\!\!\!\! r(A^{(1,2)}BC^{(1,2)}) =  \max \{ \,0, \
r(A) + r(B) + r(C) - r[\, A, \, B\,] - r[\, B^{\ast}, \, C^{\ast}\,] \}.
\label{gg19}
\end{align}
\end{lemma}

We also use the following results to establish matrix equalities that involve generalized inverses.

\begin{lemma}  [\cite{Tian:2002,Tian:2004}] \label{T12}
Let $ A \in {\mathbb C}^{m \times n}, \, B \in {\mathbb C}^{m \times k},$ $C \in {\mathbb C}^{l \times n},$ and
$D \in {\mathbb C}^{l \times k}.$ Then the following results hold$.$
\begin{enumerate}
\item[{\rm (a)}] There exists an $A^{(1)}$ such that $CA^{(1)}B  = D$ holds if and only if
\begin{align}
\R(D) \subseteq \R(C), \ \ \R(D^{\ast}) \subseteq \R(B^{\ast}),  \ and \
r\!\begin{bmatrix}  A  &  B  \\ C & D  \end{bmatrix} = r\!\begin{bmatrix}  A  \\
C  \end{bmatrix}  + r[ \, A, \,  B \,]  - r(A);
\label{145}
\end{align}
$CA^{(1)}B  = D$ holds for all $A^{(1)}$ if and only if
\begin{align}
[\, C, \, D \,] =0 \ or \ \begin{bmatrix}  B \\ D \end{bmatrix} =0  \ or \ r\!\begin{bmatrix} A  & B \\ C  & D
\end{bmatrix} = r(A).
\label{146}
\end{align}

\item[{\rm (b)}]
$CA^{(1,2)}B  = D$ holds for all $A^{(1,2)}$ if and only if
\begin{align}
\begin{bmatrix}  A  & 0 \\ 0 & D \end{bmatrix}  = 0 \ or \  [\, C, \, D \,] = 0 \
or \ \begin{bmatrix} B \\ D  \end{bmatrix} =0  \ or \
 r\!\begin{bmatrix}  A  & B \\ C  & D \end{bmatrix} = r(A).
\label{147}
\end{align}

\item[{\rm (c)}] $CA^{(1,3)}B  = D$ holds for all $A^{(1,3)}$ if and only if
\begin{align}
\begin{bmatrix} B \\ D \end{bmatrix} =0 \ or \ r\!\begin{bmatrix}
 A^{*}A & A^{*}B  \\ C & D \end{bmatrix} = r(A).
\label{148}
\end{align}

\item[{\rm (d)}] $CA^{(1,4)}B  = D$ holds for all $A^{(1,4)}$ if and only if
\begin{align}
[\, C, \, D\,] =0 \ or \  r\!\begin{bmatrix}
 AA^{*} & B  \\ CA^{*} & D \end{bmatrix} = r(A).
\label{149}
\end{align}

\item[{\rm (e)}] 
$CA^{(1,2,3)}B  = D$ holds for all $A^{(1,2,3)}$ if and only if
\begin{align}
\begin{bmatrix} A^{*}B \\ D \end{bmatrix} =0 \ \ or \ \  r\!\begin{bmatrix}
 A^{*}A & A^{*}B  \\ C & D \end{bmatrix} = r(A).
\label{150}
\end{align}

\item[{\rm (f)}] 
$CA^{(1,2,4)}B  = D$ holds for all $A^{(1,2,4)}$ if and only if
\begin{align}
[\, CA^{*}, \, D \,] =0 \ \ or \ \ r\!\begin{bmatrix}
 AA^{*} & B  \\ CA^{*} & D  \end{bmatrix} = r(A).
\label{151}
\end{align}

\item[{\rm (g)}] 
$CA^{(1,3,4)}B  = D$ holds for all $A^{(1,3,4)}$ if and only if
\begin{align}
 r\!\begin{bmatrix}
 A^{*}A & A^{*}B  \\ C & D  \end{bmatrix} =r(A)  \ or \
r\!\begin{bmatrix}
 AA^{*} & B  \\ CA^{*} & D  \end{bmatrix} = r(A).
\label{152}
\end{align}

\item[{\rm (h)}] $CA^{\dag}B = D$ holds if and only if
\begin{align}
r\!\begin{bmatrix} A^{*}AA^{*} & A^{*}B \\ CA^{*} & D
\end{bmatrix} = r(A).
\label{153}
\end{align}
\end{enumerate}
\end{lemma}

Based on Lemma \ref{T11}(g), we give a group of general rules to derive set inclusions for generalized inverses of matrices.

\begin{lemma} \label{LE27}
Let $M \in {\mathbb C}^{m \times n},$  and  let $f(A_1^{(i_1,\ldots,j_1)}, A_2^{(i_2,\ldots,j_2)}, \ldots,  A_k^{(i_k,\ldots,j_k)})  \in {\mathbb C}^{n \times m}$ be a matrix expression composed by $A^{(i_1,\ldots,j_1)}$, $A_2^{(i_2,\ldots,j_2)}$, $\ldots$, $A_k^{(i_k,\ldots,j_k)}.$ Then  the following results hold$.$
\begin{enumerate}
\item[{\rm (a)}] $\{\, M^{(1)} \} \supseteq \left\{f(A_1^{(i_1,\ldots,j_1)}, \ldots,  A_k^{(i_k,\ldots,j_k)}) \right\}$ if and only if
$$
\max_{A_1^{(i_1,\ldots,j_1)}, \ldots,  A_k^{(i_k,\ldots,j_k)}} r\!\left[M - Mf(A_1^{(i_1,\ldots,j_1)}, \ldots,  A_k^{(i_k,\ldots,j_k)})M \right] = 0.
$$

\item[{\rm (b)}] $\{\, M^{(1,2)} \} \supseteq  \left\{f(A_1^{(i_1,\ldots,j_1)}, \ldots,  A_k^{(i_k,\ldots,j_k)}) \right\}$ if and only if
$$
\left\{\!\!\begin{array}{l} \max_{A_1^{(i_1,\ldots,j_1)}, \ldots,  A_k^{(i_k,\ldots,j_k)}} r\!\left[M - Mf(A_1^{(i_1,\ldots,j_1)}, \ldots,  A_k^{(i_k,\ldots,j_k)})M \right] = 0,
\\
\max_{A_1^{(i_1,\ldots,j_1)}, \ldots,  A_k^{(i_k,\ldots,j_k)}} r\left[f(A_1^{(i_1,\ldots,j_1)}, \ldots,  A_k^{(i_k,\ldots,j_k)}) \right] = r(M).
\end{array} \right.
$$

\item[{\rm (c)}] $\{\, M^{(1,3)} \} \supseteq \left\{f(A_1^{(i_1,\ldots,j_1)}, \ldots,  A_k^{(i_k,\ldots,j_k)})  \right\}
$ if and only if
$$
\max_{A_1^{(i_1,\ldots,j_1)}, \ldots,  A_k^{(i_k,\ldots,j_k)}} r\left[M^* - M^*Mf(A_1^{(i_1,\ldots,j_1)}, \ldots,  A_k^{(i_k,\ldots,j_k)}) \right]  = 0.
$$

\item[{\rm (d)}] $\{\, M^{(1,4)} \} \supseteq \left\{f(A_1^{(i_1,\ldots,j_1)}, \ldots, A_k^{(i_k,\ldots,j_k)})  \right\}$
iff
$$
\max_{A_1^{(i_1,\ldots,j_1)}, \ldots,  A_k^{(i_k,\ldots,j_k)}}r\left[M^* - Mf(A_1^{(i_1,\ldots,j_1)}, \ldots,  A_k^{(i_k,\ldots,j_k)})MM^* \right]  = 0.
$$

\item[{\rm (e)}] $\{\, M^{(1,2,3)} \} \supseteq \left\{f(A_1^{(i_1,\ldots,j_1)}, \ldots,  A_k^{(i_k,\ldots,j_k)})  \right\}$  if and only if $\{\, M^{(1,2)} \} \supseteq \left\{f(A_1^{(i_1,\ldots,j_1)}, \ldots,  A_k^{(i_k,\ldots,j_k)})  \right\}$ and $\{\, M^{(1,3)} \} \supseteq \left\{f(A_1^{(i_1,\ldots,j_1)}, \ldots,  A_k^{(i_k,\ldots,j_k)})  \right\},$ or equivalently$,$
$$
\left\{\!\!\begin{array}{l}
\max_{A_1^{(i_1,\ldots,j_1)}, \ldots,  A_k^{(i_k,\ldots,j_k)}} r\!\left[M^* - M^*Mf(A_1^{(i_1,\ldots,j_1)}, \ldots,  A_k^{(i_k,\ldots,j_k)})M \right] = 0,
\\
\max_{A_1^{(i_1,\ldots,j_1)}, \ldots,  A_k^{(i_k,\ldots,j_k)}} r\left[f(A_1^{(i_1,\ldots,j_1)}, \ldots,  A_k^{(i_k,\ldots,j_k)}) \right] = r(M).
\end{array} \right.
$$

\item[{\rm (f)}] $\{\, M^{(1,2,4)} \} \supseteq \left\{f(A_1^{(i_1,\ldots,j_1)}, \ldots,  A_k^{(i_k,\ldots,j_k)})  \right\}$ if and only if
  $\{\, M^{(1,2)} \} \supseteq \left\{f(A_1^{(i_1,\ldots,j_1)}, \ldots,  A_k^{(i_k,\ldots,j_k)})  \right\}$ and $\{\, M^{(1,4)} \} \supseteq \left\{f(A_1^{(i_1,\ldots,j_1)}, \ldots,  A_k^{(i_k,\ldots,j_k)})  \right\},$ or equivalently$,$
 $$
\left\{\!\!\begin{array}{l} \max_{A_1^{(i_1,\ldots,j_1)}, \ldots,  A_k^{(i_k,\ldots,j_k)}} r\!\left[M^* - f(A_1^{(i_1,\ldots,j_1)}, \ldots,  A_k^{(i_k,\ldots,j_k)})MM^* \right] = 0,
\\
\max_{A_1^{(i_1,\ldots,j_1)}, \ldots,  A_k^{(i_k,\ldots,j_k)}} r\!\left[f(A_1^{(i_1,\ldots,j_1)}, \ldots,  A_k^{(i_k,\ldots,j_k)}) \right] = r(M).
\end{array} \right.
$$

\item[{\rm (g)}] $\{\, M^{(1,3,4)} \} \supseteq \left\{f(A_1^{(i_1,\ldots,j_1)}, \ldots,  A_k^{(i_k,\ldots,j_k)})  \right\}$  if and only if
$\{\, M^{(1,3)} \} \supseteq \left\{f(A_1^{(i_1,\ldots,j_1)}, \ldots,  A_k^{(i_k,\ldots,j_k)})  \right\}$  and $\{\, M^{(1,4)} \} \supseteq \left\{f(A_1^{(i_1,\ldots,j_1)}, \ldots,  A_k^{(i_k,\ldots,j_k)})  \right\}.$

\item[{\rm (h)}] $M^{\dag} = \left\{f(A_1^{(i_1,\ldots,j_1)}, \ldots,  A_k^{(i_k,\ldots,j_k)})  \right\}$ if and only if $\{\, M^{(1,2,3)} \} \supseteq \left\{f(A_1^{(i_1,\ldots,j_1)}, \ldots,  A_k^{(i_k,\ldots,j_k)})  \right\}$ and $\{\, M^{(1,2,4)} \} \supseteq \left\{f(A_1^{(i_1,\ldots,j_1)}, \ldots,  A_k^{(i_k,\ldots,j_k)})  \right\}.$
\end{enumerate}
\end{lemma}

\section[3]{Various mixed ROLs for two or three matrix products}

We first give two groups of set inclusion associated with the matrix sets $\{(AB)^{(1)}\}$ and $\{(AB)^{(1,2)}\}$.

\begin{theorem} \label{T31}
 Let $ A \in {\mathbb C}^{ m \times n}$ and $ B \in {\mathbb C}^{ n \times p}$ be given$.$  Then$,$
 \begin{enumerate}
\item[{\rm (a)}] the following set inclusions hold
\begin{align}
\{\,(AB)^{(1)} \} & \supseteq  \{ \, (A^{(1)}AB)^{(1)}A^{(1)} \},
\label{t11}
\\
\{\,(AB)^{(1)} \} & \supseteq \{\, B^{(1)}(ABB^{(1)})^{(1)} \},
\label{t12}
\\
\{\,(AB)^{(1)} \} & \supseteq \{ \, (A^{\ast}AB)^{(1)}A^{\ast} \},
\label{t13}
\\
\{\,(AB)^{(1)} \} & \supseteq \{\, B^{\ast}(ABB^{\ast})^{(1)} \},
\label{t14}
\\
\{\,(AB)^{(1)} \} & \supseteq \{ \, (AA^{\ast}AB)^{(1)}AA^{\ast} \},
\label{t15}
\\
\{\,(AB)^{(1)} \} & \supseteq \{\, B^{\ast}B(ABB^{\ast}B)^{(1)} \},
\label{t16}
\\
\{\,(AB)^{(1)} \} & \supseteq \{\, B^{(1)}(A^{(1)}ABB^{(1)})^{(1)}A^{(1)} \},
\label{t17}
\\
\{\,(AB)^{(1)} \} & \supseteq \{\, B^{\ast}(A^{\ast}ABB^{\ast})^{(1)}A^{\ast} \},
\label{t18}
\\
\{\,(AB)^{(1)} \} & \supseteq \{\, B^{\ast}B(AA^{\ast}ABB^{\ast}B)^{(1)}AA^{\ast} \};
\label{t19}
\end{align}

\item[{\rm (b)}] the following set inclusions hold
\begin{align}
\{\,(AB)^{(1,2)} \} & \supseteq  \{ \, (A^{(1,2)}AB)^{(1,2)}A^{(1,2)} \},
\label{t112}
\\
\{\,(AB)^{(1,2)} \} & \supseteq \{\, B^{(1,2)}(ABB^{(1,2)})^{(1,2)} \},
\label{t113}
\\
\{\,(AB)^{(1,2)} \} & \supseteq \{ \, (A^{\ast}AB)^{(1,2)}A^{\ast} \},
\label{t114}
\\
\{\,(AB)^{(1,2)} \} & \supseteq \{\, B^{\ast}(ABB^{\ast})^{(1,2)} \},
\label{t115}
\\
\{\,(AB)^{(1,2)} \} & \supseteq \{ \, (AA^{\ast}AB)^{(1,2)}AA^{\ast} \},
\label{t116}
\\
\{\,(AB)^{(1,2)} \} & \supseteq \{\, B^{\ast}B(ABB^{\ast}B)^{(1,2)} \},
\label{t117}
\\
\{\,(AB)^{(1,2)} \} & \supseteq \{\, B^{(1,2)}(A^{(1,2)}ABB^{(1,2)})^{(1,2)}A^{(1,2)} \},
\label{t118}
\\
\{\,(AB)^{(1,2)} \} & \supseteq \{\, B^{\ast}(A^{\ast}ABB^{\ast})^{(1,2)}A^{\ast} \},
\label{t119}
\\
\{\,(AB)^{(1,2)} \} & \supseteq \{\, B^{\ast}B(AA^{\ast}ABB^{\ast}B)^{(1,2)}AA^{\ast} \}.
\label{t120}
\end{align}
\end{enumerate}
\end{theorem}

\begin{proof}
The whole proofs are based on the definitions of generalized inverses and direct verifications.
 For any generalized inverses $A^{(1)}$,  $B^{(1)}$,  $(A^{(1)}AB)^{(1)}$,  $(ABB^{(1)})^{(1)}$,  $(A^{\ast}AB)^{(1)}$, $(ABB^{\ast})^{(1)}$, $(A^{(1)}ABB^{(1)})^{(1)}$, $(A^{\ast}ABB^{\ast})^{(1)}$, and $(AA^{\ast}ABB^{\ast}A)^{(1)}$,  it is easy to verify by definition and
 Lemma \ref{T11}(a) that
\begin{align*}
& AB[\,(A^{(1)}AB)^{(1)}A^{(1)} \,]AB  = A(A^{(1)}AB)(A^{(1)}AB)^{(1)}(A^{(1)}AB) =
 A(A^{(1)}AB) = AB,
\\
& AB[\, B^{(1)}(ABB^{(1)})^{(1)} \,]AB = (ABB^{(1)})(ABB^{(1)})^{(1)}(ABB^{(1)})B  =
 (ABB^{(1)})B = AB,
\\
& AB[\,(A^{\ast}AB)^{(1)}A^{\ast} \,]AB  = (A^{\dag})^{\ast}(A^{\ast}AB)(A^{\ast}AB)^{(1)}(A^{\ast}AB) =
 (A^{\dag})^{\ast}(A^{\ast}AB) = AB,
\\
& AB[\, B^{\ast}(ABB^{\ast})^{(1)} \,]AB = (ABB^{\ast})(ABB^{\ast})^{(1)}(ABB^{\ast})(B^{\dag})^{\ast}
= (ABB^{\ast})(B^{\dag})^{\ast}  = AB,
\\
& AB[AA^{\ast}AB)^{(1)}AA^{\ast}]AB =  [(AA^{\ast})^{\dag}]^{\ast}(AA^{\ast}AB)(AA^{\ast}AB)^{(1)}(AA^{\ast}AB) =
[(AA^{\ast})^{\dag}]^{\ast}AA^{\ast}AB  = AB,
\\
& ABB^{\ast}B(ABB^{\ast}B)^{(1)}AB = (ABB^{\ast}B)(ABB^{\ast}B)^{(1)}(ABB^{\ast}B)[(B^{\ast}B)^{\dag}]^{\ast} = ABB^{\ast}B[(B^{\ast}B)^{\dag}]^{\ast} = AB,
\end{align*}
\begin{align*}
 AB[\,B^{(1)}(A^{(1)}ABB^{(1)})^{(1)} A^{(1)}\,]AB & = A(A^{(1)}ABB^{(1)})(A^{(1)}ABB^{(1)})^{(1)}(A^{(1)}ABB^{(1)})B
 \\
& = A(A^{(1)}ABB^{(1)})B = AB,
\\
 AB[\,B^{\ast}(A^{\ast}ABB^{\ast})^{(1)}A^{\ast}\,]AB & = (A^{\dag})^{\ast}(A^{\ast}ABB^{\ast})(A^{\ast}ABB^{\ast})^{(1)}(A^{\ast}ABB^{\ast})
(B^{\dag})^{\ast} (A^{\dag})^{\ast}(A^{\ast}ABB^{\ast})(B^{\dag})^{\ast}
\\
&  = AB,
\\
 AB[\,B^{\ast}B(AA^{\ast}ABB^{\ast}B)^{(1)}AA^{\ast}\,]AB & = [(AA^{\ast})^{\dag}]^{\ast}(AA^{\ast}ABB^{\ast}B)(AA^{\ast}ABB^{\ast}B)^{(1)}(AA^{\ast}ABB^{\ast}B)
[(B^{\ast}B)^{\dag}]^{\ast}
\\
& = [(AA^{\ast})^{\dag}]^{\ast}(AA^{\ast}ABB^{\ast}B)[(B^{\ast}B)^{\dag}]^{\ast} = AB,
\end{align*}
establishing \eqref{t11}--\eqref{t19}.

Since $\{\,M^{(1)} \} \supseteq \{\,M^{(1,2)} \}$ for any matrix $M$, it is easy to see from \eqref{t11}--\eqref{t19}
that
\begin{align}
\{\,(AB)^{(1)} \} & \supseteq  \{ \, (A^{(1)}AB)^{(1)}A^{(1)} \} \supseteq  \{ \, (A^{(1,2)}AB)^{(1,2)}A^{(1,2)} \},
\label{kk122}
\\
\{\,(AB)^{(1)} \} & \supseteq \{\, B^{(1)}(ABB^{(1)})^{(1)} \} \supseteq \{\, B^{(1,2)}(ABB^{(1,2)})^{(1,2)} \},
\label{kk123}
\\
\{\,(AB)^{(1)} \} & \supseteq \{ \, (A^{\ast}AB)^{(1)}A^{\ast} \} \supseteq \{ \, (A^{\ast}AB)^{(1,2)}A^{\ast} \},
\label{kk124}
\\
\{\,(AB)^{(1)} \} & \supseteq \{\, B^{\ast}(ABB^{\ast})^{(1)} \}\supseteq \{\, B^{\ast}(ABB^{\ast})^{(1,2)} \},
\label{kk125}
\\
\{\,(AB)^{(1)} \} & \supseteq \{ \, (AA^{\ast}AB)^{(1)}AA^{\ast} \}\supseteq \{ \, (AA^{\ast}AB)^{(1,2)}AA^{\ast} \},
\label{kk126}
\\
\{\,(AB)^{(1)} \} & \supseteq \{\, B^{\ast}B(ABB^{\ast}B)^{(1)} \} \supseteq \{\, B^{\ast}B(ABB^{\ast}B)^{(1,2)} \},
\label{kk127}
\\
\{\,(AB)^{(1)} \} & \supseteq \{\, B^{(1)}(A^{(1)}ABB^{(1)})^{(1)}A^{(1)} \}  \supseteq \{\, B^{(1,2)}(A^{(1,2)}ABB^{(1,2)})^{(1,2)}A^{(1,2)} \},
\label{kk128}
\\
\{\,(AB)^{(1)} \} & \supseteq \{\, B^{\ast}(A^{\ast}ABB^{\ast})^{(1)}A^{\ast} \}  \supseteq \{\, B^{\ast}(A^{\ast}ABB^{\ast})^{(1,2)}A^{\ast} \},
\label{kk129}
\\
\{\,(AB)^{(1)} \} & \supseteq \{\, B^{\ast}B(AA^{\ast}ABB^{\ast}B)^{(1,2)}AA^{\ast} \}\supseteq \{\, B^{\ast}B(AA^{\ast}ABB^{\ast}B)^{(1,2)}AA^{\ast} \}
\label{kk130}
\end{align}
hold. Also by definition,
\begin{align}
& (A^{(1,2)}AB)^{(1,2)}A^{(1,2)}AB(A^{(1,2)}AB)^{(1,2)}A^{(1,2)} =(A^{(1,2)}AB)^{(1,2)}A^{(1,2)},
\label{kk132}
\\
& B^{(1,2)}(ABB^{(1,2)})^{(1,2)}AB B^{(1,2)}(ABB^{(1,2)})^{(1,2)} = B^{(1,2)}(ABB^{(1,2)})^{(1,2)},
\label{kk133}
\\
& (A^{\ast}AB)^{(1,2)}A^{\ast}AB(A^{\ast}AB)^{(1,2)} =  (A^{\ast}AB)^{(1,2)},
\label{kk134}
\\
& B^{\ast}(ABB^{\ast})^{(1,2)}AB B^{\ast}(ABB^{\ast})^{(1,2)} = B^{\ast}(ABB^{\ast})^{(1,2)},
\label{kk135}
\\
&(AA^{\ast}AB)^{(1,2)}AA^{\ast}AB(AA^{\ast}AB)^{(1,2)}AA^{\ast} =(AA^{\ast}AB)^{(1,2)}AA^{\ast},
\label{kk136}
\\
& B^{\ast}B(ABB^{\ast}B)^{(1,2)}ABB^{\ast}B(ABB^{\ast}B)^{(1,2)} = B^{\ast}B(ABB^{\ast}B)^{(1,2)},
\label{kk137}
\\
& B^{(1,2)}(A^{(1,2)}ABB^{(1,2)})^{(1,2)}A^{(1,2)}ABB^{(1,2)}(A^{(1,2)}ABB^{(1,2)})^{(1,2)}A^{(1,2)}
= B^{(1,2)}(A^{(1,2)}ABB^{(1,2)})^{(1,2)}A^{(1,2)},
\label{kk138}
\\
& B^{\ast}(A^{\ast}ABB^{\ast})^{(1,2)}A^{\ast}ABB^{\ast}(A^{\ast}ABB^{\ast})^{(1,2)}A^{\ast} = B^{\ast}(A^{\ast}ABB^{\ast})^{(1,2)}A^{\ast},
\label{kk139}
\\
& B^{\ast}B(AA^{\ast}ABB^{\ast}B)^{(1,2)}AA^{\ast}ABB^{\ast}B(AA^{\ast}ABB^{\ast}B)^{(1,2)}AA^{\ast} = B^{\ast}B(AA^{\ast}ABB^{\ast}B)^{(1,2)}AA^{\ast}.
\label{kk140}
\end{align}
Combining \eqref{kk122}--\eqref{kk130} with \eqref{kk132}--\eqref{kk140} leads to \eqref{t112}--\eqref{t120}.
\end{proof}

\begin{theorem} \label{T32}
 Let $ A \in \mathbb C^{ m \times n}$ and $ B \in \mathbb C^{ n \times p}$  be given$.$
 and denote $P = I_n - A^{(1)}A,$ $Q = I_n - BB^{(1)},$  $ U= I_n - A^{(1,2)}A,$ and $V = I_n - BB^{(1,2)}.$  Then$,$
  \begin{enumerate}
\item[{\rm (a)}] the set inclusion $\{\,(AB)^{(1)} \}  \supseteq \{\, B^{(1)}A^{(1)} -  B^{(1)}P(QP)^{(1)}QA^{(1)} \, \}$ always holds$;$

\item[{\rm (b)}] $\{\,(AB)^{(1,2)} \} \supseteq \{\, B^{(1,2)}A^{(1,2)} -  B^{(1,2)}U(VU)^{(1,2)}VA^{(1,2)}
\} \Leftrightarrow \ r(AB) = r(A) = r(B)  \Leftrightarrow \ \R(AB) = \R(A)$ and $\R[(AB)^{\ast}] = \R(B^{\ast}).$
\end{enumerate}
\end{theorem}

\begin{proof}
Noting $QP(QP)^{(1)}QP = QP$ and pre- and post-multiplying $A$ and $B$, we obtain
\begin{align}
AQP(QP)^{(1)}QPB = AQPB,
\label{kk143}
\end{align}
where both sides are given by
\begin{align*}
AQPB & = A(I_n - BB^{(1)})(I_n - A^{(1)}A)B = A(I_n  - A^{(1)}A - BB^{(1)}  +  BB^{(1)}A^{(1)}A)B
\\
&  = AB - ABB^{(1)}A^{(1)}AB,
\\
AQP(QP)^{(1)}QPB & = A(I_n - BB^{(1)})P(QP)^{(1)}Q(I_n - A^{(1)}A)B
\\
& = AP(QP)^{(1)}QB  - AP(QP)^{(1)}QA^{(1)}AB -  ABB^{(1)}P(QP)^{(1)}QB
\\
& \ \ \ + ABB^{(1)}P(QP)^{(1)}QA^{(1)}AB
\\
& = ABB^{(1)}P(QP)^{(1)}QA^{(1)}AB.
\end{align*}
Substituting these equalities into \eqref{kk143} yields $AB[\, B^{(1)}A^{(1)} - B^{(1)}P(QP)^{(1)}QA^{(1)} \,]AB = AB$. Thus (a) holds.

Since $\{\,M^{(1)} \} \supseteq \{\,M^{(1,2)} \}$ for any matrix $M$,
thus it is easy to see from \eqref{kk140} that
\begin{align}
\{\,(AB)^{(1)} \} \supseteq \{\, B^{(1)}A^{(1)} -  B^{(1)}U(VU)^{(1)}VA^{(1)} \}
\supseteq \{\, B^{(1,2)}A^{(1,2)} -  B^{(1,2)}U(VU)^{(1,2)}VA^{(1,2)} \,\}.
\label{kk144}
\end{align}
We next determine the maximum and minimum rank of $B^{(1,2)}A^{(1,2)} -  B^{(1,2)}U(VU)^{(1,2)}VA^{(1,2)}$.
Since $r[(VU)^{(1,2)}] = r(VU)$ and $VU(VU)^{(1,2)}VU = VU$,  it follows that
$r[U(VU)^{(1,2)}V] = r[VU(VU)^{(1,2)}VU] = r(VU)$. Also note that
$(U(VU)^{(1,2)})^2 = U(VU)^{(1,2)}$, namely, $U(VU)^{(1,2)}V$ is idempotent, we obtain
\begin{align}
r[\,I_n - U(VU)^{(1,2)}V\,] = n - r[U(VU)^{(1,2)}V] = n - r[(VU)^{(1,2)}] = n - r(VU)
= r(A) + r(B) - r(AB)
\label{kk144a}
\end{align}
holds for all $(VU)^{(1,2)}$ by \eqref{s17}. By \eqref{r12} and elementary block matrix operations,
\begin{align}
r[\, B, \, I_n - U(VU)^{(1,2)}V \,] & = r(B) + r[\,V - (VU)(VU)^{(1,2)}V\,] = r(B) + r[\,VU, \, V] - r(VU) \nb
\\
& = r(B) + (V) - r(VU) = n - r(VU) = r(A) + r(B) - r(AB),
\label{kk144b}
\\
 r\!\begin{bmatrix} A \\  I_n - U(VU)^{(1,2)}V \end{bmatrix} & =  r(A) +
r[\,U - U(VU)^{(1,2)}(VU)\,] = r(A) + r\!\begin{bmatrix} U \\ VU \end{bmatrix}  - r(VU) \nb
\\
& = r(A) + r(U) - r(VU) = n - r(VU)  = r(A) + r(B) - r(AB).
\label{kk144c}
\end{align}
Next by \eqref{gg18}, \eqref{gg19}, \eqref{kk144a}, \eqref{kk144b}, and \eqref{kk144c},
\begin{align}
& \max_{A^{(1,2)}, B^{(1,2)}}r[B^{(1,2)}A^{(1,2)} -  B^{(1,2)}U(VU)^{(1,2)}VA^{(1,2)}]  = \min_{A^{(1,2)}, B^{(1,2)}}r[B^{(1,2)}(I_n - U(VU)^{(1,2)}V)A^{(1,2)}] \nb
\\
& = \min\{r(A^{(1,2)}), \ r(B^{(1,2)}), \ r(I_n - U(VU)^{(1,2)}V) \}  = \min\{r(A), \ r(B), \ r(A) + r(B) - r(AB) \} \nb
\\
& = \min\{r(A), \ r(B) \},
\label{kk145}
\end{align}
and
\begin{align}
& \min_{A^{(1,2)}, B^{(1,2)}}r[B^{(1,2)}A^{(1,2)} -  B^{(1,2)}U(VU)^{(1,2)}VA^{(1,2)}] \nb
\\
&= \min_{A^{(1,2)}, B^{(1,2)}}r[B^{(1,2)}(I_n - U(VU)^{(1,2)}V)A^{(1,2)}] \nb
\\
& = \max\left\{0, \ r(A^{(1,2)}) + r(B^{(1,2)}) +  r[I_n - U(VU)^{(1,2)}V] - r[\, B, \, I_n - U(VU)^{(1,2)}V\,] -  r\!\begin{bmatrix} A \\  I_n - U(VU)^{(1,2)}V \end{bmatrix} \right\} \nb
\\
& = \max\{0,  \ r(AB)\} = r(AB).
\label{kk146}
\end{align}
Combining \eqref{kk145} and \eqref{kk146}, we see that $r[B^{(1,2)}A^{(1,2)} -  B^{(1,2)}U(VU)^{(1,2)}VA^{(1,2)}] = r(AB)$ holds for all $A^{(1,2)}$, $B^{(1,2)}$, and $(VU)^{(1,2)}$ if and only if $r(AB) = r(A) = r(B)$. Combining this fact with \eqref{kk144} and applying \eqref{g2}, we obtain (b).
\end{proof}

The mixed ROL $(AB)^{\dag} =
B^{\dag}A^{\dag} -  B^{\dag}[(I_n - BB^{\dag})(I_n - A^{\dag}A)]^{\dag}A^{\dag}$ for the Moore--Penrose inverses was proposed and  approached by the present author in \cite{Tian:2004b} using the matrix rank methodology.

For a triple matrix product,  there are a large variety of mixed-type reverse-order laws that can  be formulated mostly by try and fail method. Here we present such a list as follows.

\begin{theorem} \label{T33}
Let $A \in \mathbb C^{ m \times n},$
$B \in \mathbb C^{ n \times p},$ and $C \in \mathbb C^{p \times q}$ be given$,$
 and denote $M =ABC.$   Then$,$
 \begin{enumerate}
\item[{\rm (a)}] the following set inclusions hold
\begin{align*}
\{\,M^{(1)} \} & \supseteq \{\, (A^{(1)}M)^{(1)}A^{(1)} \},
\\
\{\,M^{(1)} \} & \supseteq  \{\, C^{(1)}(MC^{(1)})^{(1)} \},
\\
\{\,M^{(1)} \} & \supseteq \{ \, (A^{\ast}M)^{(1)}A^{\ast} \},
\\
\{\,M^{(1)} \} & \supseteq  \{\, C^{\ast}(MC^{\ast})^{(1)} \},
\\
\{\,M^{(1)} \} & \supseteq \{\, (AA^{\ast}M)^{(1)}AA^{\ast} \},
\\
\{\,M^{(1)} \} & \supseteq  \{\, C^{\ast}C(MC^{\ast}C)^{(1)} \},
\\
\{\,M^{(1)} \} & \supseteq  \{\, C^{(1)}(A^{(1)}MC^{(1)})^{(1)}A^{(1)} \},
\\
\{\,M^{(1)} \} & \supseteq  \{\, C^{\ast}(A^{\ast}MC^{\ast})^{(1)}A^{\ast} \},
\\
\{\,M^{(1)} \} & \supseteq \{\, [\,(AB)^{(1)}M \,]^{(1)}(AB)^{(1)} \},
\\
\{\,M^{(1)} \} & \supseteq \{\, (BC)^{(1)}[\,M(BC)^{(1)} \,]^{(1)} \},
\\
\{\,M^{(1)} \} & \supseteq \{\, [\,(AB)^{\ast}M \,]^{(1)}(AB)^{\ast} \},
\\
\{\,M^{(1)} \} & \supseteq  \{\,(BC)^{\ast}[\,M(BC)^{\ast} \,]^{(1)} \};
\\
\{\,M^{(1)} \} & \supseteq  \{ \, [\,(ABB^{(1)})^{(1)}M \,]^{(1)}(ABB^{(1)})^{(1)} \},
\end{align*}
\begin{align*}
\{\,M^{(1)} \} & \supseteq  \{\, (B^{(1)}BC)^{(1)}[\, M(B^{(1)}BC)^{(1)} \,]^{(1)} \},
\\
\{\,M^{(1)} \} & \supseteq \{\, [\,(ABB^{\ast})^{(1)}M \,]^{(1)}(ABB^{\ast})^{(1)} \},
\\
\{\,M^{(1)} \} & \supseteq \{\, (B^{\ast}BC)^{(1)}[\,M(B^{\ast}BC)^{(1)} \,]^{(1)} \},
\\
\{\,M^{(1)} \} & \supseteq  \{\,  C^{\ast}C(AA^{\ast}MC^{\ast}C)^{(1)}AA^{\ast}  \},
\\
\{\,M^{(1)} \} & \supseteq \{\,(BC)^{(1)}[\,(AB)^{(1)}M(BC)^{(1)} \,]^{(1)}(AB)^{(1)} \},
\\
\{\,M^{(1)} \} & \supseteq \{\,(BC)^{\ast}[\,(AB)^{\ast}M(BC)^{\ast} \,]^{(1)}(AB)^{\ast} \},
\\
\{\,M^{(1)} \} & \supseteq  \{\,(B^{(1)}BC)^{(1)}[\,(ABB^{(1)})^{(1)}M(B^{(1)}BC)^{(1)} \,]^{(1)}(ABB^{(1)})^{(1)} \},
\\
\{\,M^{(1)} \} & \supseteq  \{\,(B^{\ast}BC)^{(1)}[\,(ABB^{\ast})^{(1)}M(B^{\ast}BC)^{(1)} \,]^{(1)}(ABB^{\ast})^{(1)} \},
\\
\{\,M^{(1)} \} & \supseteq  \{\,(B^{(1)}BC)^{\ast}[\,(ABB^{(1)})^{\ast}M(B^{(1)}BC)^{\ast} \,]^{(1)}(ABB^{(1)})^{\ast} \},
\\
\{\,M^{(1)} \} & \supseteq  \{\,(B^{\ast}BC)^{\ast}[\,(ABB^{\ast})^{\ast}M(B^{\ast}BC)^{\ast} \,]^{(1)}(ABB^{\ast})^{\ast} \},
\end{align*}

\item[{\rm (b)}] the following set inclusions hold
\begin{align*}
\{\,M^{(1,2)} \} & \supseteq \{\, (A^{(1,2)}M)^{(1,2)}A^{(1,2)} \},
\\
\{\,M^{(1,2)} \} & \supseteq  \{\, C^{(1,2)}(MC^{(1,2)})^{(1,2)} \},
\\
\{\,M^{(1,2)} \} & \supseteq \{ \, (A^{\ast}M)^{(1,2)}A^{\ast} \},
\\
\{\,M^{(1,2)} \} & \supseteq  \{\, C^{\ast}(MC^{\ast})^{(1,2)} \},
\\
\{\,M^{(1,2)} \} & \supseteq \{\, (AA^{\ast}M)^{(1,2)}AA^{\ast} \},
\\
\{\,M^{(1,2)} \} & \supseteq  \{\, C^{\ast}C(MC^{\ast}C)^{(1,2)} \},
\\
\{\,M^{(1,2)} \} & \supseteq  \{\, C^{(1,2)}(A^{(1,2)}MC^{(1,2)})^{(1,2)}A^{(1,2)} \},
\\
\{\,M^{(1,2)} \} & \supseteq  \{\, C^{\ast}(A^{\ast}MC^{\ast})^{(1,2)}A^{\ast} \},
\\
\{\,M^{(1,2)} \} & \supseteq \{\, [\,(AB)^{(1,2)}M \,]^{(1,2)}(AB)^{(1,2)} \},
\\
\{\,M^{(1,2)} \} & \supseteq \{\, (BC)^{(1,2)}[\,M(BC)^{(1,2)} \,]^{(1,2)} \},
\\
\{\,M^{(1,2)} \} & \supseteq \{\, [\,(AB)^{\ast}M \,]^{(1,2)}(AB)^{\ast} \},
\\
\{\,M^{(1,2)} \} & \supseteq  \{\,(BC)^{\ast}[\,M(BC)^{\ast} \,]^{(1,2)} \},
\\
\{\,M^{(1,2)} \} & \supseteq  \{ \, [\,(ABB^{(1,2)})^{(1,2)}M \,]^{(1,2)}(ABB^{(1,2)})^{(1,2)} \},
\\
\{\,M^{(1,2)} \} & \supseteq  \{\, (B^{(1,2)}BC)^{(1,2)}[\, M(B^{(1,2)}BC)^{(1,2)} \,]^{(1,2)} \},
\\
\{\,M^{(1,2)} \} & \supseteq \{\, [\,(ABB^{\ast})^{(1,2)}M \,]^{(1,2)}(ABB^{\ast})^{(1,2)} \},
\\
\{\,M^{(1,2)} \} & \supseteq \{\, (B^{\ast}BC)^{(1,2)}[\,M(B^{\ast}BC)^{(1,2)} \,]^{(1,2)} \},
\\
\{\,M^{(1,2)} \} & \supseteq  \{\,  C^{\ast}C(AA^{\ast}MC^{\ast}C)^{(1,2)}AA^{\ast}  \},
\\
\{\,M^{(1,2)} \} & \supseteq \{\,(BC)^{(1,2)}[\,(AB)^{(1,2)}M(BC)^{(1,2)} \,]^{(1,2)}(AB)^{(1,2)} \},
\\
\{\,M^{(1,2)} \} & \supseteq \{\,(BC)^{\ast}[\,(AB)^{\ast}M(BC)^{\ast} \,]^{(1,2)}(AB)^{\ast} \},
\\
\{\,M^{(1,2)} \} & \supseteq  \{\,(B^{(1,2)}BC)^{(1,2)}[\,(ABB^{(1,2)})^{(1,2)}M(B^{(1,2)}BC)^{(1,2)} \,]^{(1,2)}(ABB^{(1,2)})^{(1,2)} \},
\\
\{\,M^{(1,2)} \} & \supseteq  \{\,(B^{(1,2)}BC)^{\ast}[\,(ABB^{(1,2)})^{\ast}M(B^{(1,2)}BC)^{\ast} \,]^{(1,2)}(ABB^{(1,2)})^{\ast} \},
\\
\{\,M^{(1,2)} \} & \supseteq \{\,(B^{\ast}BC)^{(1,2)}[\,(ABB^{\ast})^{(1,2)}M(B^{\ast}BC)^{(1,2)} \,]^{(1,2)}(ABB^{\ast})^{(1,2)} \},
\\
\{\,M^{(1,2)} \} & \supseteq  \{\,(B^{\ast}BC)^{\ast}[\,(ABB^{\ast})^{\ast}M(B^{\ast}BC)^{\ast} \,]^{(1,2)}(ABB^{\ast})^{\ast} \}.
\end{align*}
\end{enumerate}
\end{theorem}


\begin{proof}
It follows from the direct verification and the definitions of $\{1\}$- and  $\{1,2\}$-generalized inverse of matrices.
\end{proof}

We next prove two results related to the mixed ROLs:
\begin{align*}
 (ABC)^{(1)} & = (BC)^{(1)}B(AB)^{(1)} - (BC)^{(1)}BP(QBP)^{(1)}QB(AB)^{(1)},
\\
(ABC)^{(1,2)} & = (BC)^{(1,2)}B(AB)^{(1,2)} - (BC)^{(1,2)}BU(VBU)^{(1,2)}VB(AB)^{(1,2)}
\end{align*}
using definitions and the matrix rank formulas, where
$P = I_p - (AB)^{(1)}AB,$ $Q = I_n - BC(BC)^{(1)},$ $U = I_p - (AB)^{(1,2)}AB,$ and $V = I_n - BC(BC)^{(1,2)}.$

\begin{theorem} \label{T34}
Let $A \in \mathbb C^{ m \times n},$ $B \in \mathbb C^{ n \times p},$ and $C \in \mathbb C^{p \times q}$ be given$,$
 and denote $M =ABC.$  Then$,$
  \begin{enumerate}
\item[{\rm (a)}] $\{\,M^{(1)} \}  \supseteq \{ \, (BC)^{(1)}B(AB)^{(1)} - (BC)^{(1)}BP(QBP)^{(1)}QB(AB)^{(1)} \}$ always hold$;$

\item[{\rm (b)}] $\{\,M^{(1,2)} \} \supseteq \{ \, (BC)^{(1,2)}B(AB)^{(1,2)} - (BC)^{(1,2)}BU(VBU)^{(1,2)}VB(AB)^{(1,2)} \}
 \Leftrightarrow r(M) = r(AB) = r(BC)  \Leftrightarrow$  $\R(M) = \R(AB)$ and $\R(M^{\ast}) = \R[(BC)^{\ast}].$
 \end{enumerate}
 \end{theorem}

\begin{proof}
Noting $VBU  = (VBU)(VBU)^{(1)}(VBU)$ and pre- and post-multiplying $A$ and $C$, we obtain
\begin{align}
AVBUC  = AVBU(VBU)^{(1)}VBUC,
\label{kk196}
\end{align}
where
\begin{align*}
AVBUC & = ABC - ABC(BC)^{(1)}BC - AB(AB)^{(1)}ABC +  ABC(BC)^{(1)}B(AB)^{(1)}ABC \nb
\\
& = ABC(BC)^{(1)}B(AB)^{(1)}ABC - ABC,
\end{align*}
and
\begin{align*}
& AVBU(VBU)^{(1)}VBUC \nb
\\
 & = (A - ABC(BC)^{(1)})BU(VBU)^{(1)}VB(C - (AB)^{(1)}ABC)  \nb
\\
& = ABU(VBU)^{(1)}VBC -  ABC(BC)^{(1)}BU(VBU)^{(1)}VBC -  ABU(VBU)^{(1)}VB(AB)^{(1)}ABC
\\
& \ \ \ + ABC(BC)^{(1)}BU(VBU)^{(1)}(AB)^{(1)}ABC
\\
& = ABC(BC)^{(1)}BU(VBU)^{(1)}(AB)^{(1)}ABC.
\end{align*}
Substituting these two equalities into \eqref{kk196} yields $
M[(BC)^{(1)}B(AB)^{(1)} - (BC)^{(1)}BU(VBU)^{(1)}(AB)^{(1)}]M = M$,
establishing (a).

Result (a) obviously implies that
\begin{align}
\{\,M^{(1)} \} \supseteq \{ \, (BC)^{(1,2)}B(AB)^{(1,2)} - (BC)^{(1,2)}BU(VBU)^{(1,2)}VB(AB)^{(1,2)} \}
\label{ff245}
\end{align}
We next determine the maximum and minimum rank of $(BC)^{(1,2)}[B- BU(VBU)^{(1,2)}VB](AB)^{(1,2)}$.
By \eqref{r12} and elementary block matrix operations,
\begin{align}
& r(VB) = r[\, B - (BC)(BC)^{(1,2)}B\,] = r[\, BC, \,  B\,] - r(BC) = r(B) - r(BC),
\label{z245}
\\
& r(BU) = r[\, B - B(AB)^{(1,2)}AB\,] = r[\, (AB)^{\ast}, \, B^{\ast}\,] - r(AB) = r(B) - r(AB).
 \label{z246}
\end{align}
By \eqref{r12} and elementary block matrix operations,
\begin{align}
& \max_{(VBU)^{(1,2)}} r[\, B - BU(VBU)^{(1,2)}VB \,]  \nb
\\
& = \min  \left\{ r[\,BU, \, B \,], \ r\!\begin{bmatrix}  VB \\ B \end{bmatrix}, \
 r\!\begin{bmatrix}  VBU  & VB \\ BU  & B \end{bmatrix} -r(VBU) \right\} \nb
 \\
& =  \min\{ r(B), \  r(B) -r(VBU) \} \nb
 \\
& = r(B) - r(VBU) = r(AB) + r(AB) - r(ABC)  \ \ \mbox{(by \eqref{s18})};
\label{kk198}
\\
&  \min_{(VBU)^{(1,2)}} r[\, B - BU(VBU)^{(1,2)}VB \,] \nb
\\
& =  r\!\begin{bmatrix} VB \\ B
 \end{bmatrix} + r[\,BU, \, B\,] + r(VBU)  \nb
\\
& + \max \left\{r\!\begin{bmatrix} VBU  & VB \\ BU & B
 \end{bmatrix}\!
 - r\!\begin{bmatrix} VBU  & 0  & VB \\ 0   & BU & B \end{bmatrix}
 - r\!\begin{bmatrix} VBU  & 0 \\ 0   & VB \\ BU & B
\end{bmatrix}, \  r(B) -  \!r\!\begin{bmatrix} VBU  & 0 \\ BU  & B
 \end{bmatrix}\! -r\!\begin{bmatrix} VBU  & VB \\ 0   & B  \end{bmatrix} \right\} \nb
 \\
& =  2r(B) +  r(VBU) + \min\{ -2r(VBU) - r(B), \  \  -2r(VBU) - r(B) \} \nb
 \\
& =  r(B)-  r(VBU)  = r(AB) + r(AB) - r(ABC) \ \ \mbox{(by \eqref{s18})}.
 \label{kk199}
\end{align}
Combining  \eqref{kk198} and \eqref{kk199}, we see that
\begin{align}
r[\, B- BU(VBU)^{(1,2)}VB \,] = r(AB) + r(AB) - r(ABC)
\label{kk1100}
\end{align}
holds for all $(VBU)^{(1,2)}$.  By \eqref{gg18}, \eqref{gg19},  \eqref{z245}, \eqref{z246},
and elementary block matrix operations,
\begin{align}
&r[\,BC, \, B - BU(VBU)^{(1,2)}VB\,] =  r(BC) +  r[\, VB - (VBU)(VBU)^{(1,2)}VB\,] \nb
\\
& =  r(BC) +  r[\, VBU, \ VB\,] - r(VBU) =  r(B)- r(VBU) = r(AB) + r(AB) - r(ABC),
\label{kk1101}
\\
& r[(AB)^{\ast}, \  (B - BU(VBU)^{(1,2)}VB)]^{\ast}] = r(AB) + r(AB) - r(ABC).
\label{kk1102}
\end{align}
Also by \eqref{gg18} and \eqref{gg19},
\begin{align}
& \max_{(AB)^{(1,2)}, (BC)^{(1,2)}}r[(BC)^{(1,2)}B(AB)^{(1,2)} - (BC)^{(1,2)}BU(VBU)^{(1,2)}VB(AB)^{(1,2)}] \nb
\\
& = \max_{(AB)^{(1,2)}, (BC)^{(1,2)}}r\{(BC)^{(1,2)}[B - BU(VBU)^{(1,2)}VB](AB)^{(1,2)} \} \nb
\\
=& \max_{(AB)^{(1,2)}, \ \  (BC)^{(1,2)}}\{ r((AB)^{(1,2)}), \ r((BC)^{(1,2)}),
 \ \  r[\,B- BU(VBU)^{(1,2)}VB)\,] \} \nb
\\
& =\min\{r(AB), \ \ r(BC), \ \  r(AB) + r(ABC) - r(ABC) \}  \ \ \mbox{(by \eqref{kk1100})}\nb
\\
& =\min\{r(AB), \ \ r(BC)\},
\label{kk1103}
\end{align}
and
\begin{align}
& \min_{(AB)^{(1,2)}, (BC)^{(1,2)}}r[(BC)^{(1,2)}B(AB)^{(1,2)} - (BC)^{(1,2)}BU(VBU)^{(1,2)}VB(AB)^{(1,2)}] \nb
\\
& = \min_{(AB)^{(1,2)}, (BC)^{(1,2)}}r\{(BC)^{(1,2)}[B - BU(VBU)^{(1,2)}VB](AB)^{(1,2)}\} \nb
\\
& =\max\{0, \ \ r(AB) + r((BC) + r[B - BU(VBU)^{(1,2)}VB] \nb
 \\
& \ \ \ \ \ \ \ \  - r[\,BC, \, B - BU(VBU)^{(1,2)}VB\,] - r[(AB)^{\ast}, \,  (B - BU(VBU)^{(1,2)}VB)]^{\ast} \} \nb
\\
& =\max\{0, \ \ r(ABC) \} \ \ \mbox{(by \eqref{kk1100}, \eqref{kk1101}, and \eqref{kk1102})} \nb
\\
& = r(ABC).
\label{kk1104}
\end{align}
Combining \eqref{kk1103} and \eqref{kk1104}, we see that $r[(BC)^{(1,2)}B(AB)^{(1,2)} - (BC)^{(1,2)}BU(VBU)^{(1,2)}VB(AB)^{(1,2)}] = r(ABC)$ holds for all $(AB)^{(1,2)}$, $(BC)^{(1,2)}$, and $(VU)^{(1,2)}$ if and only if $r(ABC) = r(AB) = r(BC)$. Combining this fact with \eqref{ff245} and applying \eqref{g2} lead to (b).
\end{proof}

\section[4]{Reverse order laws for a triple matrix product with applications}

We first prepare some general formulas associated with matrix calculations in  \eqref{14}--\eqref{17}.

\begin{lemma} \label{T21}
Let $A \in {\mathbb C}^{m \times m},$ $B \in {\mathbb C}^{m\times n},$ and $C \in {\mathbb C}^{n \times n}$ be given
and assume that $A$ and $C$ are nonsingular$.$ Also denote $M = ABC.$ Then$,$
\begin{enumerate}
\item[{\rm (a)}] The following rank equalities
\begin{align}
& r(M) = r(B), \ \  r(CM^{(i,\ldots,j)}A) = r(M^{(i,\ldots,j)}), \ \  r(C^{-1}B^{(k,\ldots,l)}A^{-1}) =  r(B^{(k,\ldots,l)})
\label{21}
\end{align}
hold for all $M^{(i,\ldots,j)}$ and $B^{(k,\ldots,l)}.$

\item[{\rm (b)}] The following set inclusions hold
\begin{align}
& C^{-1}B^{\dag}A^{-1} \in \{C^{-1}B^{(1,3,4)}A^{-1}\} \subseteq  \{C^{-1}B^{(1,4)}A^{-1}\}  \subseteq  \{C^{-1}B^{(1)}A^{-1}\},
\label{22}
\\
& C^{-1}B^{\dag}A^{-1} \in \{C^{-1}B^{(1,3,4)}A^{-1}\} \subseteq  \{C^{-1}B^{(1,3)}A^{-1}\}  \subseteq  \{C^{-1}B^{(1)}A^{-1}\},
\label{23}
\\
& C^{-1}B^{\dag}A^{-1} \in \{C^{-1}B^{(1,2,4)}A^{-1}\} \subseteq  \{C^{-1}B^{(1,4)}A^{-1}\}  \subseteq  \{C^{-1}B^{(1)}A^{-1}\},
\label{24}
\\
& C^{-1}B^{\dag}A^{-1} \in \{C^{-1}B^{(1,2,4)}A^{-1}\} \subseteq  \{C^{-1}B^{(1,2)}A^{-1}\}  \subseteq  \{C^{-1}B^{(1)}A^{-1}\},
\label{25}
\\
& C^{-1}B^{\dag}A^{-1} \in \{C^{-1}B^{(1,2,3)}A^{-1}\} \subseteq  \{C^{-1}B^{(1,3)}A^{-1}\}  \subseteq  \{C^{-1}B^{(1)}A^{-1}\},
\label{26}
\\
&  C^{-1}B^{\dag}A^{-1} \in \{C^{-1}B^{(1,2,3)}A^{-1}\} \subseteq  \{C^{-1}B^{(1,2)}A^{-1}\}  \subseteq  \{C^{-1}B^{(1)}A^{-1}\}.
\label{27}
\end{align}

\item[{\rm (c)}] The following equivalent facts hold
\begin{align}
& \{M^{(i,\ldots,j)}\} \cap \{C^{-1}B^{(k,\ldots,l)}A^{-1}\} \neq \emptyset
\Leftrightarrow \{CM^{(i,\ldots,j)}A\} \cap \{B^{(k,\ldots,l)}\} \neq \emptyset,
\label{28}
\\
&\{M^{(i,\ldots,j)}\} \supseteq  \{C^{-1}B^{(k,\ldots,l)}A^{-1}\} \Leftrightarrow \{CM^{(i,\ldots,j)}A\} \supseteq  \{B^{(k,\ldots,l)}\},
\label{29}
\\
& \{M^{(i,\ldots,j)}\} \subseteq  \{C^{-1}B^{(k,\ldots,l)}A^{-1}\} \Leftrightarrow \{CM^{(i,\ldots,j)}A\} \subseteq  \{B^{(k,\ldots,l)}\},
\label{210}
\\
& \{M^{(i,\ldots,j)}\} = \{C^{-1}B^{(k,\ldots,l)}A^{-1}\} \Leftrightarrow \{CM^{(i,\ldots,j)}A\} = \{B^{(k,\ldots,l)}\}
\label{211}
\end{align}
hold for the eight common-used generalized inverses of $B$ and $M.$

\item[{\rm (d)}] The following equalities
\begin{align}
MC^{-1}B^{(1)}A^{-1}M & = MC^{-1}B^{(1,2)}A^{-1}M = MC^{-1}B^{(1,3)}A^{-1}M \nb
\\
& = MC^{-1}B^{(1,4)}A^{-1}M = MC^{-1}B^{(1,2,3)}A^{-1}M \nb
\\
& = MC^{-1}B^{(1,2,4)}A^{-1}M = MC^{-1}B^{(1,3,4)}A^{-1}M \nb
\\
& = MC^{-1}B^{\dag}A^{-1}M  = M
\label{212}
\end{align}
hold for all the eight commonly-used types of generalized inverses of $B$, and
\begin{align}
BCM^{(1)}AB & = BCM^{(1,2)}AB = BCM^{(1,3)}AB = BCM^{(1,4)}AB \nb
\\
&= BCM^{(1,2,3)}AB = BCM^{(1,2,4)}AB = BCM^{(1,3,4)}AB \nb
\\
& = BCM^{\dag}AB = B
\label{213}
\end{align}
hold for all the eight commonly-used types of generalized inverses of $M$.

\item[{\rm (e)}] $M^{\ast}MC^{-1}B^{(1)}A^{-1} = M^{\ast}MC^{-1}B^{(1,2)}A^{-1} = M^{\ast}MC^{-1}B^{(1,4)}A^{-1} = M^{\ast}MC^{-1}B^{(1,2,4)}A^{-1} =M^{\ast}$
hold for some $B^{(1)}$,  $B^{(1,2)}$, $B^{(1,4)}$, and $B^{(1,2,4)}$.

\item[{\rm (f)}] $M^{\ast}MC^{-1}B^{(1)}A^{-1} = M^{\ast}MC^{-1}B^{(1,2)}A^{-1} = M^{\ast}MC^{-1}B^{(1,4)}A^{-1} = M^{\ast}MC^{-1}B^{(1,2,4)}A^{-1} =M^{\ast}$ hold for all $B^{(1)},$ $B^{(1,2)},$ $B^{(1,4)},$  and $B^{(1,2,4)}$  $\Leftrightarrow$  $B =0$  or $r(B) = m.$

\item[{\rm (g)}] $M^{\ast}MC^{-1}B^{(1,3)}A^{-1} = M^{\ast}MC^{-1}B^{(1,2,3)}A^{-1} = M^{\ast}MC^{-1}B^{(1,3,4)}A^{-1} = M^{\ast}MC^{-1}B^{\dag}A^{-1} = M^{\ast} \Leftrightarrow  \R(A^{*}AB) = \R(B).$

\item[{\rm (h)}] $B^{\ast}BCM^{(1)}A = B^{\ast}BCM^{(1,2)}A = B^{\ast}BCM^{(1,4)}A  = B^{\ast}BCM^{(1,2,4)}A  =B^{\ast}$ hold for all $M^{(1)}$,  $M^{(1,2)}$, $M^{(1,4)},$ and $M^{(1,2,4)}$ $\Leftrightarrow$ $B =0$ or $r(B) = m.$

\item[{\rm (i)}] $B^{\ast}BCM^{(1,3)}A = B^{\ast}BCM^{(1,2,3)}A =  B^{\ast}BCM^{(1,3,4)}A  = B^{\ast}BCM^{\dag}A  =B^{\ast} \Leftrightarrow  \R(A^{*}AB) = \R(B).$
\end{enumerate}
\end{lemma}

\begin{proof}
Result (a) follows from the nonsingularity of $A$ and $C$. Pre- and post-multiplying \eqref{s1}--\eqref{s6} with $C^{-1}$ and $A^{-1}$ respectively yield the set inclusions in Result (b).  Pre- and post-multiplying \eqref{14}--\eqref{17} with $C$ and $A$ respectively yield the following equivalent facts in Result (c).  Result (d) follows from direct verification.

The following results
\begin{align}
& M^{\ast}MC^{-1}B^{(1)}A^{-1} = M^{\ast}MC^{-1}B^{(1,2)}A^{-1} = M^{\ast}MC^{-1}B^{(1,4)}A^{-1} = M^{\ast}MC^{-1}B^{(1,2,4)}A^{-1} =M^{\ast} \nb
\\
& \Leftrightarrow  (AB)^{\ast}ABB^{(1)} = (AB)^{\ast}ABB^{(1,2)} = (AB)^{\ast}ABB^{(1,4)} = (AB)^{\ast}ABB^{(1,2,4)} = (AB)^{\ast}A,
\label{214}
\end{align}
and
\begin{align}
& M^{\ast}MC^{-1}B^{(1,3)}A^{-1} = M^{\ast}MC^{-1}B^{(1,2,3)}A^{-1} = M^{\ast}MC^{-1}B^{(1,3,4)}A^{-1} = M^{\ast}MC^{-1}B^{\dag}A^{-1} = M^{\ast} \nb
\\
& \Leftrightarrow  (AB)^{\ast}ABB^{\dag}  = (AB)^{\ast}A
\label{215}
\end{align}
follow from \eqref{120}--\eqref{123} and the nonsingularity of $A$ and $C$. Furthermore, It is easy to derive from Lemma \ref{T12} that
\begin{align}
& (AB)^{\ast}ABB^{(1)} = (AB)^{\ast}ABB^{(1,2)} = (AB)^{\ast}ABB^{(1,4)} = (AB)^{\ast}ABB^{(1,2,4)} = (AB)^{\ast}A \nb
\\
& \mbox{are solvable for some}  \ B^{(1)}, \ B^{(1,2)}, \ B^{(1,4)}, \ {\rm and} \ B^{(1,2,4)},
\label{216}
\end{align}
combining it with \eqref{214} leads to Result (e);
\begin{align}
& (AB)^{\ast}ABB^{(1)} = (AB)^{\ast}ABB^{(1,2)} = (AB)^{\ast}ABB^{(1,4)} = (AB)^{\ast}ABB^{(1,2,4)} = (AB)^{\ast}A \nb
\\
& \mbox{are solvable for all}  \ B^{(1)}, \ B^{(1,2)}, \ B^{(1,4)}, \ {\rm and} \ B^{(1,2,4)} \Leftrightarrow
 B =0 \ {\rm or} \ r(B) = m,
\label{217}
\end{align}
combining it with \eqref{214} leads to Result (f);
\begin{align}
& (AB)^{\ast}ABB^{\dag}  = (AB)^{\ast}A \Leftrightarrow  \R(A^{*}AB) = \R(B).
\label{218}
\end{align}
combining it with \eqref{215} leads to Result (g).

Results (h) and (i) are also derived from Lemma \ref{T12}.
\end{proof}

Armed with the preceding results and facts, we can derive the main in the paper.  For the sake of convenience of reference, we will present a complete list of results for all the situations in \eqref{14}--\eqref{17}.

\begin{theorem} \label{T22}
Let $A \in {\mathbb C}^{m \times m},$ $B \in {\mathbb C}^{m\times n},$ and $C \in {\mathbb C}^{n \times n}$ be given
and assume that $A$ and $C$ are  nonsingular$.$ Also denote $M = ABC.$ Then the following 64 groups of result hold$.$
\begin{enumerate}
\item[{\rm (1)}] $\{M^{(1)}\} = \{C^{-1}B^{(1)}A^{-1}\}$ holds$.$

\item[{\rm (2a)}] $\{M^{(1)}\} \supseteq \{C^{-1}B^{(1,2)}A^{-1}\}$  holds$.$

 \item[{\rm (2b)}] $\{M^{(1)}\} \subseteq \{C^{-1}B^{(1,2)}A^{-1}\}$ $\Leftrightarrow$  $\{M^{(1)}\} = \{C^{-1}B^{(1,2)}A^{-1}\}$ $\Leftrightarrow$
     $r(B) = \min\{m, \, n\}.$

\item[{\rm (3a)}] $\{M^{(1)}\} \supseteq \{C^{-1}B^{(1,3)}A^{-1}\}$ holds$.$

\item[{\rm (3b)}] $\{M^{(1)}\} \subseteq \{C^{-1}B^{(1,3)}A^{-1}\}$ $\Leftrightarrow$
$\{M^{(1)}\} =\{C^{-1}B^{(1,3)}A^{-1}\}$ $\Leftrightarrow$ $B =0$ or $r(B) = m.$

\item[{\rm (4a)}] $\{M^{(1)}\} \supseteq \{C^{-1}B^{(1,4)}A^{-1}\}$ holds$.$

\item[{\rm (4b)}] $\{M^{(1)}\} \subseteq \{C^{-1}B^{(1,4)}A^{-1}\}$ $\Leftrightarrow$
$\{M^{(1)}\} = \{C^{-1}B^{(1,4)}A^{-1}\}$ $\Leftrightarrow$ $B =0$ or $r(B) = n.$

\item[{\rm (5a)}] $\{M^{(1)}\} \supseteq \{C^{-1}B^{(1,2,3)}A^{-1}\}$ holds$.$

\item[{\rm (5b)}] $\{M^{(1)}\} \subseteq \{C^{-1}B^{(1,2,3)}A^{-1}\}$ $\Leftrightarrow$
$\{M^{(1)}\} = \{C^{-1}B^{(1,2,3)}A^{-1}\}$ $\Leftrightarrow$ $r(B) = m.$

\item[{\rm (6a)}] $\{M^{(1)}\} \supseteq \{C^{-1}B^{(1,2,4)}A^{-1}\}$ holds$.$

\item[{\rm (6b)}] $\{M^{(1)}\} \subseteq \{C^{-1}B^{(1,2,4)}A^{-1}\}$ $\Leftrightarrow$
$\{M^{(1)}\} = \{C^{-1}B^{(1,2,4)}A^{-1}\}$ $\Leftrightarrow$ $r(B) = n.$

\item[{\rm (7a)}] $\{M^{(1)}\} \supseteq \{C^{-1}B^{(1,3,4)}A^{-1}\}$ holds$.$

\item[{\rm (7b)}] $\{M^{(1)}\} \subseteq \{C^{-1}B^{(1,3,4)}A^{-1}\}$ $\Leftrightarrow$
$\{M^{(1)}\} = \{C^{-1}B^{(1,3,4)}A^{-1}\}$ $\Leftrightarrow$
 $B =0$ or $r(B) = m =n.$

\item[{\rm (8)}] $\{M^{(1)}\} \ni C^{-1}B^{\dag}A^{-1}$ holds$.$

\item[{\rm (9a)}] $\{M^{(1,2)}\} \subseteq \{C^{-1}B^{(1)}A^{-1}\}$ holds$.$

\item[{\rm (9b)}]  $\{M^{(1,2)}\} \supseteq \{C^{-1}B^{(1)}A^{-1}\}$ $\Leftrightarrow$
$\{M^{(1,2)}\} = \{C^{-1}B^{(1)}A^{-1}\}$ $\Leftrightarrow$ $r(B) = m$ or $r(B) = n.$

\item[{\rm (10)}] $\{M^{(1,2)}\} = \{C^{-1}B^{(1,2)}A^{-1}\}$ holds$.$

\item[{\rm (11a)}] $\{M^{(1,2)}\} \cap \{C^{-1}B^{(1,3)}A^{-1}\}\neq \emptyset$ holds$.$

\item[{\rm (11b)}]  $\{M^{(1,2)}\} \supseteq \{C^{-1}B^{(1,3)}A^{-1}\}$ $\Leftrightarrow$ $r(B) = m$ or $r(B) = n.$

\item[{\rm (11c)}] $\{M^{(1,2)}\} \subseteq \{C^{-1}B^{(1,3)}A^{-1}\}$  $\Leftrightarrow$ $B =0$ or $r(B) = m.$

\item[{\rm (11d)}] $\{M^{(1,2)}\}  = \{C^{-1}B^{(1,3)}A^{-1}\}$  $\Leftrightarrow$ $r(B) = m.$

\item[{\rm (12a)}] $\{M^{(1,2)}\} \cap \{C^{-1}B^{(1,4)}A^{-1}\}\neq \emptyset.$ holds$.$

\item[{\rm (12b)}]  $\{M^{(1,2)}\} \supseteq \{C^{-1}B^{(1,4)}A^{-1}\}$ $\Leftrightarrow$ $r(B) = m$ or $r(B) = n.$

\item[{\rm (12c)}] $\{M^{(1,2)}\} \subseteq \{C^{-1}B^{(1,4)}A^{-1}\}$ $\Leftrightarrow$
 $B =0$ or $r(B) = n.$

\item[{\rm (12d)}] $\{M^{(1,2)}\}  = \{C^{-1}B^{(1,4)}A^{-1}\}$ $\Leftrightarrow$ $r(B) = n.$

\item[{\rm (13a)}]  $\{M^{(1,2)}\} \supseteq \{C^{-1}B^{(1,2,3)}A^{-1}\}$ holds$.$

\item[{\rm (13b)}] $\{M^{(1,2)}\} \subseteq \{C^{-1}B^{(1,2,3)}A^{-1}\}$ $\Leftrightarrow$
$\{M^{(1,2)}\} = \{C^{-1}B^{(1,2,3)}A^{-1}\}$ $\Leftrightarrow$  $B =0$ or $r(B) = m.$

\item[{\rm (14a)}]  $\{M^{(1,2)}\} \supseteq \{C^{-1}B^{(1,2,4)}A^{-1}\}$ holds$.$

\item[{\rm (14c)}] $\{M^{(1,2)}\} \subseteq \{C^{-1}B^{(1,2,4)}A^{-1}\}$ $\Leftrightarrow$
 $\{M^{(1,2)}\} = \{C^{-1}B^{(1,2,4)}A^{-1}\}$ $\Leftrightarrow$ $B =0$ or $r(B) = n.$

 \item[{\rm (15a)}] $\{M^{(1,2)}\} \cap \{C^{-1}B^{(1,3,4)}A^{-1}\}\neq \emptyset$ holds$.$

\item[{\rm (15b)}]  $\{M^{(1,2)}\} \supseteq \{C^{-1}B^{(1,3,4)}A^{-1}\}$ $\Leftrightarrow$ $r(B) = m$ or $r(B) = n.$

\item[{\rm (15c)}] $\{M^{(1,2)}\} \subseteq \{C^{-1}B^{(1,3,4)}A^{-1}\}$ $\Leftrightarrow$ $B =0$ or $r(B) = m =n.$

\item[{\rm (15d)}] $\{M^{(1,2)}\} = \{C^{-1}B^{(1,3,4)}A^{-1}\}$ $\Leftrightarrow$ $r(B) = m =n.$

\item[{\rm (16)}] $\{M^{(1,2)}\} \ni C^{-1}B^{\dag}A^{-1}$ holds$.$

\item[{\rm (17a)}] $\{M^{(1,3)}\} \subseteq \{C^{-1}B^{(1)}A^{-1}\}$ holds$.$

\item[{\rm (17b)}]  $\{M^{(1,3)}\} \supseteq \{C^{-1}B^{(1)}A^{-1}\}$ $\Leftrightarrow$
$\{M^{(1,3)}\} = \{C^{-1}B^{(1)}A^{-1}\}$ $\Leftrightarrow$  $B =0$ or $r(B) = m.$

 \item[{\rm (18a)}] $\{M^{(1,3)}\} \cap \{C^{-1}B^{(1,2)}A^{-1}\}\neq \emptyset$ holds$.$

\item[{\rm (18b)}]  $\{M^{(1,3)}\} \supseteq \{C^{-1}B^{(1,2)}A^{-1}\}$ $\Leftrightarrow$ $B =0$ or $r(B) = m.$

\item[{\rm (18c)}] $\{M^{(1,3)}\} \subseteq \{C^{-1}B^{(1,2)}A^{-1}\}$ $\Leftrightarrow$ $r(B) = m$ or $r(B) = n.$

\item[{\rm (18d)}] $\{M^{(1,3)}\} = \{C^{-1}B^{(1,2)}A^{-1}\}$ $\Leftrightarrow$ $r(B) = m.$

\item[{\rm (19a)}] $\{M^{(1,3)}\} \cap \{C^{-1}B^{(1,3)}A^{-1}\}\neq \emptyset$ holds$.$

\item[{\rm (19b)}] $\{M^{(1,3)}\} \supseteq  \{C^{-1}B^{(1,3)}A^{-1}\}$  $\Leftrightarrow$  $\{M^{(1,3)}\} \subseteq  \{C^{-1}B^{(1,3)}A^{-1}\}$  $\Leftrightarrow$  $\{M^{(1,3)}\} = \{C^{-1}B^{(1,3)}A^{-1}\}$
    $\Leftrightarrow$ $\R(A^{*}AB) = \R(B).$

\item[{\rm (20a)}] $\{M^{(1,3)}\} \cap \{C^{-1}B^{(1,4)}A^{-1}\}\neq \emptyset$ holds$.$

\item[{\rm (20b)}] $\{M^{(1,3)}\} \supseteq \{C^{-1}B^{(1,4)}A^{-1}\}$ $\Leftrightarrow$ $B =0$ or $r(B) = m.$

\item[{\rm (20c)}] $\{M^{(1,3)}\} \subseteq \{C^{-1}B^{(1,4)}A^{-1}\}$ $\Leftrightarrow$ $B =0$ or $r(B) = n.$

\item[{\rm (20d)}] $\{M^{(1,3)}\} = \{C^{-1}B^{(1,4)}A^{-1}\}$ $\Leftrightarrow$ $B =0$ or $r(B) = m =n.$

\item[{\rm (21a)}] $\{M^{(1,3)}\} \cap \{C^{-1}B^{(1,2,3)}A^{-1}\} \neq \emptyset$ $\Leftrightarrow$ $\{M^{(1,3)}\} \supseteq \{C^{-1}B^{(1,2,3)}A^{-1}\}$ $\Leftrightarrow$ $\R(A^{*}AB) = \R(B).$

\item[{\rm (21b)}] $\{M^{(1,3)}\} \subseteq \{C^{-1}B^{(1,2,3)}A^{-1}\}$ $\Leftrightarrow$
$\{M^{(1,3)}\} = \{C^{-1}B^{(1,2,3)}A^{-1}\}$ $\Leftrightarrow$ $\R(A^{*}AB) = \R(B)$ and $r(B) = \min\{m, \, n\}.$

 \item[{\rm (22a)}] $\{M^{(1,3)}\} \cap \{C^{-1}B^{(1,2,4)}A^{-1}\}\neq \emptyset$ holds$.$

\item[{\rm (22b)}]  $\{M^{(1,3)}\} \supseteq \{C^{-1}B^{(1,2,4)}A^{-1}\}$ $\Leftrightarrow$ $B =0$ or $r(B) = m.$

\item[{\rm (22c)}] $\{M^{(1,3)}\} \subseteq \{C^{-1}B^{(1,2,4)}A^{-1}\}$ $\Leftrightarrow$ $r(B) = n.$

\item[{\rm (22d)}] $\{M^{(1,3)}\} = \{C^{-1}B^{(1,2,4)}A^{-1}\}$ $\Leftrightarrow$ $r(B) = m = n.$

 \item[{\rm (23a)}] $\{M^{(1,3)}\} \cap \{C^{-1}B^{(1,3,4)}A^{-1}\}\neq \emptyset$
  $\Leftrightarrow$ $\{M^{(1,3)}\}  \supseteq \{C^{-1}B^{(1,3,4)}A^{-1}\}$ $\Leftrightarrow$ $\R(A^{*}AB) = \R(B).$

\item[{\rm (23b)}] $\{M^{(1,3)}\}  \subseteq \{C^{-1}B^{(1,3,4)}A^{-1}\}\neq \emptyset$
  $\Leftrightarrow$ $\{M^{(1,3)}\}  = \{C^{-1}B^{(1,3,4)}A^{-1}\}$ $\Leftrightarrow$
   $B =0$ or $\{r(B) = n$ and $\R(A^{\ast}AB) = \R(B)\}.$

\item[{\rm (24)}] $\{M^{(1,3)}\} \ni C^{-1}B^{\dag}A^{-1}$ $\Leftrightarrow$ $\R(A^{*}AB) = \R(B).$

\item[{\rm (25a)}] $\{M^{(1,4)}\} \subseteq \{C^{-1}B^{(1)}A^{-1}\}$ holds$.$

\item[{\rm (25b)}] $\{M^{(1,4)}\} \supseteq \{C^{-1}B^{(1)}A^{-1}\}$ $\Leftrightarrow$
$\{M^{(1,3)}\} = \{C^{-1}B^{(1)}A^{-1}\}$ $\Leftrightarrow$  $B =0$ or $r(B) = n.$

 \item[{\rm (26a)}] $\{M^{(1,4)}\} \cap \{C^{-1}B^{(1,2)}A^{-1}\}\neq \emptyset$ holds$.$

\item[{\rm (26b)}]  $\{M^{(1,4)}\} \supseteq \{C^{-1}B^{(1,2)}A^{-1}\}$ $\Leftrightarrow$ $B =0$ or $r(B) = n.$

\item[{\rm (26c)}] $\{M^{(1,4)}\} \subseteq \{C^{-1}B^{(1,2)}A^{-1}\}$ $\Leftrightarrow$ $r(B) =m$ or $r(B) = n.$

\item[{\rm (26d)}] $\{M^{(1,4)}\} = \{C^{-1}B^{(1,2)}A^{-1}\}$ $\Leftrightarrow$ $r(B) = n.$

\item[{\rm (27a)}] $\{M^{(1,4)}\} \cap \{C^{-1}B^{(1,3)}A^{-1}\}\neq \emptyset$ holds$.$

\item[{\rm (27b)}] $\{M^{(1,4)}\} \supseteq \{C^{-1}B^{(1,3)}A^{-1}\}$ $\Leftrightarrow$ $B =0$ or $r(B) = n.$

\item[{\rm (27c)}] $\{M^{(1,4)}\} \subseteq \{C^{-1}B^{(1,3)}A^{-1}\}$ $\Leftrightarrow$ $B =0$ or $r(B) = m.$

\item[{\rm (27d)}] $\{M^{(1,4)}\} = \{C^{-1}B^{(1,3)}A^{-1}\}$ $\Leftrightarrow$ $B =0$ or $r(B) = m = n.$

\item[{\rm (28a)}]  $\{M^{(1,4)}\} \cap \{C^{-1}B^{(1,4)}A^{-1}\} \neq \emptyset$ holds$.$

\item[{\rm (28b)}]   $\{M^{(1,4)}\} \supseteq \{C^{-1}B^{(1,4)}A^{-1}\}$ $\Leftrightarrow$
$\{M^{(1,4)}\} \subseteq  \{C^{-1}B^{(1,4)}A^{-1}\}$ $\Leftrightarrow$$\{M^{(1,4)}\} = \{C^{-1}B^{(1,4)}A^{-1}\}$ $\Leftrightarrow$ $\R(CC^{*}B^{*}) = \R(B^{*}).$

\item[{\rm (29a)}] $\{M^{(1,4)}\} \cap \{C^{-1}B^{(1,2,3)}A^{-1}\}\neq \emptyset$  holds$.$

\item[{\rm (29b)}]  $\{M^{(1,4)}\} \supseteq \{C^{-1}B^{(1,2,3)}A^{-1}\}$ $\Leftrightarrow$ $B =0$ or $r(B) = n.$

\item[{\rm (29c)}] $\{M^{(1,4)}\} \subseteq \{C^{-1}B^{(1,2,3)}A^{-1}\}$ $\Leftrightarrow$ $r(B) = m.$

\item[{\rm (29d)}] $\{M^{(1,4)}\} = \{C^{-1}B^{(1,2,3)}A^{-1}\}$ $\Leftrightarrow$ $r(B) = m = n.$

\item[{\rm (30a)}] $\{M^{(1,4)}\} \cap \{C^{-1}B^{(1,2,4)}A^{-1}\} \neq \emptyset$ $\Leftrightarrow$ $\{M^{(1,4)}\} \supseteq \{C^{-1}B^{(1,2,4)}A^{-1}\}$ $\Leftrightarrow$ $\R(CC^{*}B^{*}) = \R(B^{*}).$

\item[{\rm (30b)}] $\{M^{(1,4)}\} \subseteq \{C^{-1}B^{(1,2,4)}A^{-1}\}$ $\Leftrightarrow$
$\{M^{(1,4)}\} = \{C^{-1}B^{(1,2,4)}A^{-1}\}$ $\Leftrightarrow$ $\R(CC^{*}B^{*}) = \R(B^{*})$ and
$r(B) = \min\{m, \, n\}.$

 \item[{\rm (31a)}] $\{M^{(1,4)}\} \cap \{C^{-1}B^{(1,3,4)}A^{-1}\}\neq \emptyset$
  $\Leftrightarrow$ $\{M^{(1,4)}\}  \supseteq \{C^{-1}B^{(1,3,4)}A^{-1}\}$ $\Leftrightarrow$ $\R(CC^{*}B^{*}) = \R(B^{*}.$

\item[{\rm (31b)}] $\{M^{(1,4)}\}  \subseteq \{C^{-1}B^{(1,3,4)}A^{-1}\}$  $\Leftrightarrow$ $\{M^{(1,4)}\}  = \{C^{-1}B^{(1,3,4)}A^{-1}\}$ $\Leftrightarrow$
   $B =0$ or $\{r(B) = m$ and $\R(CC^{*}B^{*}) = \R(B^{*}\}.$

\item[{\rm (32)}] $\{M^{(1,4)}\} \ni C^{-1}B^{\dag}A^{-1}$ $\Leftrightarrow$
$\R(CC^{*}B^{*}) = \R(B^{*}).$

\item[{\rm (33a)}]  $\{M^{(1,2,3)}\} \subseteq \{C^{-1}B^{(1)}A^{-1}\}$ holds$.$

\item[{\rm (33b)}] $\{M^{(1,2,3)}\} \supseteq  \{C^{-1}B^{(1)}A^{-1}\}$ $\Leftrightarrow$
$\{M^{(1,2,3)}\} = \{C^{-1}B^{(1)}A^{-1}\}$ $\Leftrightarrow$  $r(B) = m.$

\item[{\rm (34a)}] $\{M^{(1,2,3)}\} \subseteq \{C^{-1}B^{(1,2)}A^{-1}\}$ holds$.$

\item[{\rm (34b)}]  $\{M^{(1,2,3)}\} \supseteq  \{C^{-1}B^{(1,2)}A^{-1}\}$ $\Leftrightarrow$
$\{M^{(1,2,3)}\} = \{C^{-1}B^{(1,2)}A^{-1}\}$ $\Leftrightarrow$ $B =0$ or $r(B) = m.$

\item[{\rm (35a)}]  $\{M^{(1,2,3)}\} \cap \{C^{-1}B^{(1,3)}A^{-1}\} \neq \emptyset$
  $\Leftrightarrow$  $\{M^{(1,2,3)}\} \subseteq \{C^{-1}B^{(1,3)}A^{-1}\}$ $\Leftrightarrow$ $\R(A^{*}AB) = \R(B).$

\item[{\rm (35b)}] $\{M^{(1,2,3)}\} \supseteq \{C^{-1}B^{(1,3)}A^{-1}\}$
  $\Leftrightarrow$  $\{M^{(1,2,3)}\} = \{C^{-1}B^{(1,3)}A^{-1}\}$
  $\Leftrightarrow$ $\R(A^{*}AB) = \R(B)$ and $r(B) = \min\{m, \, n\}.$

\item[{\rm (36a)}]  $\{M^{(1,2,3)}\} \cap \{C^{-1}B^{(1,4)}A^{-1}\} \neq \emptyset$ holds$.$

\item[{\rm (36b)}]  $\{M^{(1,2,3)}\} \supseteq  \{C^{-1}B^{(1,4)}A^{-1}\}$ $\Leftrightarrow$  $r(B) = m.$

\item[{\rm (36c)}]  $\{M^{(1,2,3)}\} \subseteq \{C^{-1}B^{(1,4)}A^{-1}\}$ $\Leftrightarrow$  $B =0$ or $r(B) = n.$

\item[{\rm (36d)}]  $\{M^{(1,2,3)}\} = \{C^{-1}B^{(1,4)}A^{-1}\}$ $\Leftrightarrow$ $r(B) = m =n.$

\item[{\rm (37)}]  $\{M^{(1,2,3)}\} \cap \{C^{-1}B^{(1,2,3)}A^{-1}\} \neq \emptyset$  $\Leftrightarrow$
$\{M^{(1,2,3)}\} = \{C^{-1}B^{(1,2,3)}A^{-1}\}$ $\Leftrightarrow$ $\R(A^{*}AB) = \R(B).$

\item[{\rm (38a)}]  $\{M^{(1,2,3)}\} \cap \{C^{-1}B^{(1,2,4)}A^{-1}\} \neq \emptyset$ holds$.$

\item[{\rm (38b)}]  $\{M^{(1,2,3)}\} \supseteq  \{C^{-1}B^{(1,2,4)}A^{-1}\}$ $\Leftrightarrow$ $B =0$ or $r(B) = m.$

\item[{\rm (38c)}]  $\{M^{(1,2,3)}\} \subseteq \{C^{-1}B^{(1,2,4)}A^{-1}\}$ $\Leftrightarrow$  $B =0$ or $r(B) = n.$

\item[{\rm (38d)}]  $\{M^{(1,2,3)}\} = \{C^{-1}B^{(1,2,4)}A^{-1}\}$ $\Leftrightarrow$  $B =0$ or $r(B) = m = n.$

\item[{\rm (39a)}]  $\{M^{(1,2,3)}\} \cap \{C^{-1}B^{(1,3,4)}A^{-1}\} \neq \emptyset$  $\Leftrightarrow$ $\R(A^{*}AB) = \R(B).$

\item[{\rm (39b)}]  $\{M^{(1,2,3)}\} \supseteq  \{C^{-1}B^{(1,3,4)}A^{-1}\}$ $\Leftrightarrow$ $\R(A^{*}AB) = \R(B)$ and $r(B) = \min\{m, \, n\}.$

\item[{\rm (39c)}]  $\{M^{(1,2,3)}\} \subseteq  \{C^{-1}B^{(1,3,4)}A^{-1}\}$ $\Leftrightarrow$ $B =0$ or
 $\{\R(A^{*}AB) = \R(B) \ and \  r(B) = n\}.$

\item[{\rm (39d)}]  $\{M^{(1,2,3)}\} \subseteq  \{C^{-1}B^{(1,3,4)}A^{-1}\}$ $\Leftrightarrow$
$\R(A^{*}AB) = \R(B)$ and $r(B) = n.$

\item[{\rm (40)}] $\{M^{(1,2,3)}\} \ni C^{-1}B^{\dag}A^{-1}$ $\Leftrightarrow$ $\R(A^{*}AB) = \R(B).$

\item[{\rm (41a)}] $\{M^{(1,2,4)}\} \subseteq \{C^{-1}B^{(1)}A^{-1}\}$ holds$.$

\item[{\rm (41b)}]  $\{M^{(1,2,4)}\} \supseteq  \{C^{-1}B^{(1)}A^{-1}\}$  $\Leftrightarrow$
 $\{M^{(1,2,4)}\} = \{C^{-1}B^{(1)}A^{-1}\}$  $\Leftrightarrow$ $r(B) = n.$

\item[{\rm (42a)}] $\{M^{(1,2,4)}\} \subseteq \{C^{-1}B^{(1,2)}A^{-1}\}$ holds$.$

\item[{\rm (42b)}]  $\{M^{(1,2,4)}\} \supseteq  \{C^{-1}B^{(1,2)}A^{-1}\}$  $\Leftrightarrow$
$\{M^{(1,2,4)}\} = \{C^{-1}B^{(1,2)}A^{-1}\}$  $\Leftrightarrow$
$\Leftrightarrow$  $B= 0$ or $r(B) = n.$

\item[{\rm (43a)}]  $\{M^{(1,2,4)}\} \cap \{C^{-1}B^{(1,3)}A^{-1}\} \neq \emptyset$ holds$.$

\item[{\rm (43b)}]  $\{M^{(1,2,4)}\} \supseteq  \{C^{-1}B^{(1,3)}A^{-1}\}$ $\Leftrightarrow$ $r(B) = n.$

\item[{\rm (43c)}]  $\{M^{(1,2,4)}\} \subseteq \{C^{-1}B^{(1,3)}A^{-1}\}$ $\Leftrightarrow$  $B =0$ or $r(B) = m.$

\item[{\rm (43d)}]  $\{M^{(1,2,4)}\} = \{C^{-1}B^{(1,3)}A^{-1}\}$ $\Leftrightarrow$  $r(B) = m =n.$

\item[{\rm (44a)}]  $\{M^{(1,2,4)}\} \cap \{C^{-1}B^{(1,4)}A^{-1}\} \neq \emptyset$
  $\Leftrightarrow$  $\{M^{(1,2,4)}\} \subseteq \{C^{-1}B^{(1,4)}A^{-1}\}$ $\Leftrightarrow$
  $\R(CC^{*}B^{*}) = \R(B^{*}).$

\item[{\rm (44b)}] $\{M^{(1,2,4)}\} \supseteq \{C^{-1}B^{(1,4)}A^{-1}\}$
  $\Leftrightarrow$  $\{M^{(1,2,4)}\} = \{C^{-1}B^{(1,4)}A^{-1}\}$
  $\Leftrightarrow$ $\R(CC^{*}B^{*}) = \R(B^{*})$ and $r(B) = \min\{m, \, n\}.$

\item[{\rm (45a)}]  $\{M^{(1,2,4)}\} \cap \{C^{-1}B^{(1,2,3)}A^{-1}\} \neq \emptyset$ holds$.$

\item[{\rm (45b)}]  $\{M^{(1,2,4)}\} \supseteq  \{C^{-1}B^{(1,2,3)}A^{-1}\}$ $\Leftrightarrow$ $B =0$ or $r(B) = n.$

\item[{\rm (45c)}]  $\{M^{(1,2,4)}\} \subseteq \{C^{-1}B^{(1,2,3)}A^{-1}\}$ $\Leftrightarrow$  $B =0$ or $r(B) = m.$

\item[{\rm (45d)}]  $\{M^{(1,2,4)}\} = \{C^{-1}B^{(1,2,3)}A^{-1}\}$ $\Leftrightarrow$  $B =0$ or $r(B) = m = n.$

\item[{\rm (46)}]  $\{M^{(1,2,4)}\} \cap \{C^{-1}B^{(1,2,4)}A^{-1}\} \neq \emptyset$  $\Leftrightarrow$
$\{M^{(1,2,4)}\} = \{C^{-1}B^{(1,2,4)}A^{-1}\}$ $\Leftrightarrow$ $\R(CC^{*}B^{*}) = \R(B^{*}).$

\item[{\rm (47a)}]  $\{M^{(1,2,4)}\} \cap \{C^{-1}B^{(1,3,4)}A^{-1}\} \neq \emptyset$  $\Leftrightarrow$
$\R(CC^{*}B^{*}) = \R(B^{*}).$

\item[{\rm (47b)}]  $\{M^{(1,2,4)}\} \supseteq  \{C^{-1}B^{(1,3,4)}A^{-1}\}$ $\Leftrightarrow$
$\R(CC^{*}B^{*}) = \R(B^{*})$ and $r(B) = \min\{m, \, n\}.$

\item[{\rm (47c)}] $\{M^{(1,2,4)}\} \subseteq \{C^{-1}B^{(1,3,4)}A^{-1}\}$ $\Leftrightarrow$ $B= 0$ or
$\{\R(CC^{*}B^{*}) = \R(B^{*}) \ and \ r(B) = m\}.$

\item[{\rm (47d)}] $\{M^{(1,2,4)}\} = \{C^{-1}B^{(1,3,4)}A^{-1}\}$ $\Leftrightarrow$
$\R(CC^{*}B^{*}) = \R(B^{*})$ and $r(B) = m.$

\item[{\rm (48)}] $\{M^{(1,2,4)}\} \ni C^{-1}B^{\dag}A^{-1}$ $\Leftrightarrow$ $\R(CC^{*}B^{*}) = \R(B^{*}).$

\item[{\rm (49a)}] $\{M^{(1,3,4)}\} \subseteq \{C^{-1}B^{(1)}A^{-1}\}$ holds$.$

\item[{\rm (49b)}] $\{M^{(1,3,4)}\} \supseteq \{C^{-1}B^{(1)}A^{-1}\}$ $\Leftrightarrow$
$\{M^{(1,3,4)}\} = \{C^{-1}B^{(1)}A^{-1}\}$ $\Leftrightarrow$  $B =0$ or $r(B) = m =n.$

\item[{\rm (50a)}] $\{M^{(1,3,4)}\} \cap \{C^{-1}B^{(1,2)}A^{-1}\}\neq \emptyset.$

\item[{\rm (50b)}] $\{M^{(1,3,4)}\} \supseteq \{C^{-1}B^{(1,2)}A^{-1}\}$ $\Leftrightarrow$ $B =0$ or $r(B) = m =n.$

\item[{\rm (50c)}] $\{M^{(1,3,4)}\} \subseteq \{C^{-1}B^{(1,2)}A^{-1}\}$ $\Leftrightarrow$ $r(B)= m$ or $r(B) = n.$

\item[{\rm (50d)}] $\{M^{(1,3,4)}\} = \{C^{-1}B^{(1,2)}A^{-1}\}$ $\Leftrightarrow$ $r(B) = m = n.$

\item[{\rm (51a)}] $\{M^{(1,3,4)}\} \cap \{C^{-1}B^{(1,3)}A^{-1}\}\neq \emptyset$ $\Leftrightarrow$
$\{M^{(1,3,4)}\} \subseteq \{C^{-1}B^{(1,3)}A^{-1}\}$ $\Leftrightarrow$ $\R(A^{\ast}AB) = \R(B).$

\item[{\rm (51b)}] $\{M^{(1,3,4)}\} \supseteq \{C^{-1}B^{(1,3)}A^{-1}\}$ $\Leftrightarrow$ $\{M^{(1,3,4)}\} = \{C^{-1}B^{(1,3)}A^{-1}\}$ $\Leftrightarrow$ $B =0$ or $\{r(B) = n$ and $\R(A^{\ast}AB) = \R(B)\}.$

\item[{\rm (52a)}] $\{M^{(1,3,4)}\} \cap \{C^{-1}B^{(1,4)}A^{-1}\}\neq \emptyset$ $\Leftrightarrow$
$\{M^{(1,3,4)}\} \subseteq \{C^{-1}B^{(1,4)}A^{-1}\}$ $\Leftrightarrow$ $\R(CC^{*}B^{*}) = \R(B^{*}).$

\item[{\rm (52b)}] $\{M^{(1,3,4)}\} \supseteq \{C^{-1}B^{(1,4)}A^{-1}\}$ $\Leftrightarrow$ $\{M^{(1,3,4)}\} = \{C^{-1}B^{(1,4)}A^{-1}\}$ $\Leftrightarrow$ $B =0$ or $\{r(B) = m$ and $\R(CC^{*}B^{*}) = \R(B^{*})\}.$

\item[{\rm (53a)}] $\{M^{(1,3,4)}\} \cap \{C^{-1}B^{(1,2,3)}A^{-1}\}\neq \emptyset$ $\Leftrightarrow$
$\R(A^{\ast}AB) = \R(B).$

\item[{\rm (53b)}] $\{M^{(1,3,4)}\} \supseteq \{C^{-1}B^{(1,2,3)}A^{-1}\}$ $\Leftrightarrow$
 $B =0$ or $\{r(B) = n$ and $\R(A^{\ast}AB) = \R(B)\}.$

\item[{\rm (53c)}] $\{M^{(1,3,4)}\} \subseteq \{C^{-1}B^{(1,2,3)}A^{-1}\}$ $\Leftrightarrow$ $\R(A^{\ast}AB) = \R(B)$ and $r(B) = \min\{m, \, n\}.$

\item[{\rm (53d)}] $\{M^{(1,3,4)}\} = \{C^{-1}B^{(1,2,3)}A^{-1}\}$ $\Leftrightarrow$  $r(B) = n$ and
$\R(A^{\ast}AB) = \R(B).$

\item[{\rm (54a)}] $\{M^{(1,3,4)}\} \cap \{C^{-1}B^{(1,2,4)}A^{-1}\}\neq \emptyset$ $\Leftrightarrow$
$\R(CC^{*}B^{*}) = \R(B^{*}).$

\item[{\rm (54b)}] $\{M^{(1,3,4)}\} \supseteq \{C^{-1}B^{(1,2,4)}A^{-1}\}$ $\Leftrightarrow$
 $B =0$ or $\{r(B) = m$ and $\R(CC^{*}B^{*}) = \R(B^{*})\}.$

\item[{\rm (54c)}] $\{M^{(1,3,4)}\} \subseteq \{C^{-1}B^{(1,2,4)}A^{-1}\}$ $\Leftrightarrow$ $\R(CC^{*}B^{*}) = \R(B^{*})$
and $r(B) = \min\{m, \, n\}.$

\item[{\rm (54d)}] $\{M^{(1,3,4)}\} = \{C^{-1}B^{(1,2,4)}A^{-1}\}$ $\Leftrightarrow$  $r(B) = m$ and
$\R(CC^{*}B^{*}) = \R(B^{*}).$

\item[{\rm (55)}] $\{M^{(1,3,4)}\} \cap \{C^{-1}B^{(1,3,4)}A^{-1}\} \neq \emptyset$  $\Leftrightarrow$ $\{M^{(1,3,4)}\} = \{C^{-1}B^{(1,3,4)}A^{-1}\}$  $\Leftrightarrow$ $\R(A^{\ast}AB) = \R(B)$ and $\R(CC^{*}B^{*}) = \R(B^{*}).$

\item[{\rm (56)}] $\{M^{(1,3,4)}\} \ni C^{-1}B^{\dag}A^{-1}$ $\Leftrightarrow$
$\R(A^{\ast}AB) = \R(B)$ and $\R(CC^{*}B^{*}) = \R(B^{*}).$

\item[{\rm (57)}] $M^{\dag} \in \{C^{-1}B^{(1)}A^{-1}\}$ holds$.$

\item[{\rm (58)}] $M^{\dag} \in \{C^{-1}B^{(1,2)}A^{-1}\}$ holds$.$

\item[{\rm (59)}] $M^{\dag} \in \{C^{-1}B^{(1,3)}A^{-1}\}$ $\Leftrightarrow$ $\R(A^{\ast}AB) = \R(B).$

\item[{\rm (60)}] $M^{\dag} \in \{C^{-1}B^{(1,4)}A^{-1}\}$ $\Leftrightarrow$  $\R(CC^{*}B^{*}) = \R(B^{*}).$

\item[{\rm (61)}]  $M^{\dag} \in \{C^{-1}B^{(1,2,3)}A^{-1}\}$ $\Leftrightarrow$ $\R(A^{\ast}AB) = \R(B).$

\item[{\rm (62)}] $M^{\dag} \in \{C^{-1}B^{(1,2,4)}A^{-1}\}$ $\Leftrightarrow$ $\R(CC^{*}B^{*}) = \R(B^{*}).$

\item[{\rm (63)}] $M^{\dag} \in \{C^{-1}B^{(1,3,4)}A^{-1}\}$ $\Leftrightarrow$ $\R(A^{\ast}AB) = \R(B)$ and $\R(CC^{*}B^{*}) = \R(B^{*}).$

\item[{\rm (64)}] {\rm \cite[Case 2]{Har:1986}}
 $M^{\dag} = C^{-1}B^{\dag}A^{-1}$ $\Leftrightarrow$ $\R(A^{\ast}AB) = \R(B)$ and $\R(CC^{*}B^{*}) = \R(B^{*})$ $\Leftrightarrow$ $\R(AA^{*}M) = \R(M)$ and $\R(C^{*}CM^{*}) = \R(M^{*})$ $\Leftrightarrow$ $(A^{\ast}AB)(A^{\ast}AB)^{\dag} = BB^{\dag}$ and $(BCC^{*})^{\dag}(BCC^{*}) = B^{\dag}B$
  $\Leftrightarrow$ $A^{*}ABB^{*}$ and $B^{*}BCC^{*}$ are EP $\Leftrightarrow$ $AA^{*}MM^{*}$ and $M^{*}MC^{*}C$ are EP$.$
\end{enumerate}
\end{theorem}

\begin{proof}
By \eqref{212} and \eqref{213},
\begin{align}
\{M^{(1)}\} \supseteq  \{C^{-1}B^{(1)}A^{-1}\} \  {\rm and} \ \{CM^{(1)}A\} \subseteq  \{B^{(1)}\}
\label{q219}
\end{align}
hold. Combining \eqref{q219} with \eqref{210} yields Result (1).

Result (2a) follows from \eqref{212}. By \eqref{210},
\begin{align}
& \{M^{(1)}\} \subseteq \{C^{-1}B^{(1,2)}A^{-1} \}  \Leftrightarrow \{CM^{(1)}A\} \subseteq \{B^{(1,2)} \} \nb
\\
& \Leftrightarrow BCM^{(1)}AB = B \ {\rm and} \ r(CM^{(1)}A) = r(B) \ \mbox{for all $M^{(1)}$}  \ \mbox{(by \eqref{g2})} \nb
\\
& \Leftrightarrow r(M^{(1)}) = r(B) \ \mbox{for all $M^{(1)}$}  \nb
\\
& \Leftrightarrow \min\{m, \, n\}  = r(B) \ \ \mbox{(by \eqref{135} and \eqref{136})}  \nb
\\
& \Leftrightarrow r(B) = m \ {\rm or} \ r(B) = n,
\label{q220}
\end{align}
establishing the equivalence the first and third terms in Result (2b). Combining this fact with Result (2a)
leads to the second equivalence in Result (2b).

Result (3a) follows from \eqref{212}. By \eqref{g3} and \eqref{210},
\begin{align}
& \{M^{(1)}\} \subseteq \{C^{-1}B^{(1,3)}A^{-1} \} \Leftrightarrow \{CM^{(1)}A\} \subseteq \{B^{(1,3)} \} \nb
\\
& \Leftrightarrow B^{*}BCM^{(1)}A = B^{*} \ \mbox{for all $M^{(1)}$}  \nb
\\
& \Leftrightarrow B =0  \ {\rm or} \  r(B) = m \ \mbox{(by Lemma \ref{T21}(h))},
\label{q221}
\end{align}
establishing the equivalence the first and third terms in Result (3b). Combining this fact with Result (3a)
leads to the second equivalence in Result (3b).

Result (5a) follows from \eqref{212}. By \eqref{210},
\begin{align}
& \{M^{(1)}\} \subseteq \{C^{-1}B^{(1,2,3)}A^{-1} \}  \Leftrightarrow \{CM^{(1)}A\} \subseteq \{B^{(1,2,3)} \} \nb
\\
& \Leftrightarrow B^{*}BCM^{(1)}A = B^{*} \ {\rm and} \ r(CM^{(1)}A) = r(B) \ \mbox{for all $M^{(1)}$}  \ \mbox{(by \eqref{g5})} \nb
\\
& \Leftrightarrow \{B =0 \ {\rm or} \ r(B) = m\}  \ {\rm and} \  r(B) = \min\{m, \, n\} \ \mbox{(by \eqref{135}, \eqref{136}, \eqref{21}, and Lemma \ref{T21}(h))}  \nb
\\
& \Leftrightarrow r(B) = m,
\label{q222}
\end{align}
establishing the equivalence the first and third terms in Result (5b). Combining this fact with Result (5a)
leads to the second equivalence in Result (5b).

Result (7a) follows from \eqref{212}. By \eqref{210},
\begin{align}
& \{M^{(1)}\} \subseteq \{C^{-1}B^{(1,3,4)}A^{-1} \}  \Leftrightarrow \{CM^{(1)}A\} \subseteq \{B^{(1,3,4)} \} \nb
\\
& \Leftrightarrow B^{*}BCM^{(1)}A = B^{*} \ and  \ CM^{(1)}ABB^{*} = B^{*} \mbox{for all $M^{(1,3,4)}$} \ \mbox{(by \eqref{g7})} \nb
\\
& \Leftrightarrow \{B =0 \ {\rm or} \ r(B) = m\}  \ {\rm and} \ \{B =0 \ {\rm or} \ r(B) = n\} \ \mbox{(by Lemma \ref{T21}(h))}  \nb
\\
& \Leftrightarrow B =0 \ {\rm or} \ r(B) = m = n,
\label{q223}
\end{align}
establishing the equivalence the first and third terms in Result (7b). Combining this fact with Result (7a)
leads to the second equivalence in Result (7b).

Result (8) follows from \eqref{212}.

Result (9a) follows from \eqref{s4} and Result (1). By \eqref{210},
\begin{align}
& \{M^{(1,2)}\} \supseteq \{C^{-1}B^{(1)}A^{-1} \}  \nb
\\
& \Leftrightarrow \{M^{(1)}\} \supseteq \{C^{-1}B^{(1)}A^{-1} \}
\ {\rm and} \ r(C^{-1}B^{(1)}A^{-1}) = r(M) \ \mbox{for all $B^{(1)}$}  \ \mbox{(by \eqref{g2})} \nb
\\
& \Leftrightarrow r(B^{(1)}) = r(B) \ \mbox{for all $B^{(1)}$} \ \mbox{(by \eqref{21} and \eqref{212})}  \nb
\\
& \Leftrightarrow r(B) = \min\{m, \, n\} \ \mbox{(by \eqref{135} and \eqref{136})}  \nb
\\
& \Leftrightarrow r(B) = m \ {\rm or} \ r(B) = n,
\label{q224}
\end{align}
establishing the equivalence the first and third terms in Result (9b). Combining this fact with Result (9a)
leads to the second equivalence in Result (9b).

By \eqref{21},
\begin{align}
r(C^{-1}B^{(1,2)}A^{-1}) = r(B^{(1,2)}) = r(B)
\label{q225}
\end{align}
holds for all $B^{(1,2)}$. Combining this fact with Result (2a) leads to
\begin{align}
\{M^{(1,2)}\} \supseteq \{C^{-1}B^{(1,2)}A^{-1}\}.
\label{q226}
\end{align}
On the other hand, both $BCM^{(1,2)}AB = B$ and $r(CM^{(1,2)}A) = r(M^{(1,2)}) = r(M) = r(B)$ for all $M^{(1,2)}$ hold
by \eqref{212}, which implies $\{CM^{(1,2)}A\} \subseteq \{B^{(1,2)}\}$ by \eqref{g2}, so that $\{M^{(1,2)}\} \subseteq \{C^{-1}B^{(1,2)}A^{-1}\}$ by \eqref{210}. Combining this fact with \eqref{q225} leads to the set equality in Result (10).


By \eqref{136}, there exists a $B^{(1,3)}$ such that $r(C^{-1}B^{(1,3)}A^{-1}) = r(B^{(1,3)}) = r(B) = r(M)$.
Combining this fact with \eqref{212}, we see that the product satisfies $C^{-1}B^{(1,3)}A^{-1}\in \{M^{(1,2)}\}$,
thus establishing Result (11a). By \eqref{g2},
\begin{align}
& \{M^{(1,2)}\} \supseteq  \{C^{-1}B^{(1,3)}A^{-1}\} \nb
\\
& \Leftrightarrow
MC^{-1}B^{(1,3)}A^{-1}M = M \ and \ r(C^{-1}B^{(1,3)}A^{-1})  = r(M) \ \mbox{for all $B^{(1,3)}$}  \nb
\\
& \Leftrightarrow  r(B^{(1,3)}) = r(B) \ \mbox{for all $B^{(1,3)}$} \ \mbox{(by \eqref{21} and \eqref{212})}   \nb
\\
& \Leftrightarrow \min\{m, \, n\}  = r(B) \ \mbox{(by \eqref{135} and \eqref{136})}  \nb
\\
& \Leftrightarrow r(B) = m \ {\rm or} \ r(B) = n,
\label{q227}
\end{align}
thus establishing Result (11b). By \eqref{210},
\begin{align}
& \{M^{(1,2)}\} \subseteq \{C^{-1}B^{(1,3)}A^{-1}\}  \Leftrightarrow \{CM^{(1,2)}A\} \subseteq \{B^{(1,3)}\} \nb
\\
& \Leftrightarrow B^{*}BCM^{(1,2)}A = B^{*} \ \mbox{for all $M^{(1,2)}$} \ \mbox{(by \eqref{g3})}   \nb
\\
& \Leftrightarrow B =0 \ {\rm or} \ r(B) = m \ \mbox{(by Lemma \ref{T21}(h))},
\label{q228}
\end{align}
thus establishing Result (11c). Combining  Results (11b) and (11c) leads to Result (11d).


Results (13a) follows from \eqref{27} and Result (10). By \eqref{210},
\begin{align}
& \{M^{(1,2)}\} \subseteq \{C^{-1}B^{(1,2,3)}A^{-1}\} \Leftrightarrow \{CM^{(1,2)}A\} \subseteq \{B^{(1,2,3)}\} \nb
\\
& \Leftrightarrow B^{*}BCM^{(1,2)}A = B^{*} \ and \ r(CM^{(1,2)}A)  = r(B) \ \mbox{for all $M^{(1,2)}$} \ \mbox{(by \eqref{g5})} \nb
\\
& \Leftrightarrow B =0 \ {\rm or} \ r(B) = m \ \mbox{(by Lemma \ref{T21}(h))},
\label{q229}
\end{align}
establishing the equivalence the first and third terms in Result (13b). Combining this fact with Result (13a)
leads to the second equivalence in Result (13b).


By \eqref{136}, there exists a $B^{(1,3,4)}$ such that $r(C^{-1}B^{(1,3,4)}A^{-1}) = r(B^{(1,3,4)}) = r(B) = r(M)$.
Combining this fact with \eqref{212}, we see that the product satisfies $C^{-1}B^{(1,3,4)}A^{-1}\in \{M^{(1,2)}\}$,
thus establishing Result (15a). By \eqref{g2},
\begin{align}
& \{M^{(1,2)}\} \supseteq  \{C^{-1}B^{(1,3,4)}A^{-1}\} \nb
\\
 & \Leftrightarrow MC^{-1}B^{(1,3,4)}A^{-1}M = M \ {\rm and} \ r(C^{-1}B^{(1,3,4)}A^{-1})  = r(M) \ \mbox{for all $B^{(1,3,4)}$}  \nb
\\
& \Leftrightarrow  r(B^{(1,3,4)}) = r(B) \ \mbox{for all $B^{(1,3,4)}$} \ \mbox{(by \eqref{21} and \eqref{212})}   \nb
\\
& \Leftrightarrow r(B) = \min\{m, \, n\} \ \mbox{(by \eqref{135} and \eqref{136})}  \nb
\\
& \Leftrightarrow r(B) = m \ {\rm or} \ r(B) = n,
\label{230}
\end{align}
thus establishing Result (15b). By \eqref{210},
\begin{align}
& \{M^{(1,2)}\} \subseteq \{C^{-1}B^{(1,3,4)}A^{-1}\}  \Leftrightarrow \{CM^{(1,2)}A\} \subseteq \{B^{(1,3,4)}\} \nb
\\
& \Leftrightarrow B^{*}BCM^{(1,2)}A = B^{*} \ {\rm and} \ CM^{(1,2)}ABB^{*} = B^{*} \ \mbox{for all $M^{(1,2)}$} \ \mbox{(by \eqref{g7})} \nb
\\
& \Leftrightarrow \{B =0 \ {\rm or} \ r(B) = m\}  \ {\rm and} \ \{B =0 \ {\rm or} \ r(B) = n\} \ \mbox{(by Lemma \ref{T21}(h))} \nb
\\
& \Leftrightarrow B =0 \ {\rm or} \ r(B) = m  =n,
\label{s231}
\end{align}
thus establishing Result (15c). Combining Results (11b) and (15c) leads to Result (15d).

Result (16) follows from \eqref{212}, $r(C^{-1}B^{\dag}A^{-1}) = r(B)$,  and \eqref{g2}.

Result (17a) follows from \eqref{210} and \eqref{213}. By \eqref{g3},
\begin{align}
\{M^{(1,3)}\} \supseteq \{C^{-1}B^{(1)}A^{-1}\} & \Leftrightarrow M^{*}MC^{-1}B^{(1)}A^{-1} = M^{*} \ \mbox{for all $B^{(1)}$} \nb
\\
& \Leftrightarrow  B =0 \ {\rm or} \ r(B) = m \ \mbox{(by Lemma \ref{T21}(f))},
\label{q232}
\end{align}
thus establishing the equivalence the first and third terms in Result (17b). Combining this fact with Result (17a)
leads to the second equivalence in Result (17b).

Result (18a) follows from Lemma \ref{T21}(e). By \eqref{g3},
\begin{align}
\{M^{(1,3)}\} \supseteq \{C^{-1}B^{(1,2)}A^{-1}\} & \Leftrightarrow M^{*}MC^{-1}B^{(1,2)}A^{-1} = M^{*} \
\mbox{for all $B^{(1,2)}$} \nb
\\
& \Leftrightarrow  B =0 \ {\rm or} \ r(B) = m \ \mbox{(by Lemma \ref{T21}(f))},
\label{q233}
\end{align}
thus establishing Result (18b). By \eqref{210},
\begin{align}
&\{M^{(1,3)}\} \subseteq \{C^{-1}B^{(1,2)}A^{-1}\}  \Leftrightarrow \{CM^{(1,3)}A\} \subseteq \{B^{(1,2)}\} \nb
\\
& \Leftrightarrow BCM^{(1,3)}AB = B  \ {\rm and} \ r(CM^{(1,3)}A) = r(B)  \mbox{for all $M^{(1,3)}$} \nb
\\
& \Leftrightarrow r(M^{(1,3)}) = r(B) \ \mbox{for all $M^{(1,3)}$} \ \mbox{(by \eqref{21} and \eqref{213})} \nb
\\
& \Leftrightarrow  r(B) = m \ {\rm or} \ r(B) = n \ \mbox{(by \eqref{135} and \eqref{136})},
\label{q234}
\end{align}
thus establishing Result (18c). Combining Results (18b) and (18c) leads to Result (18d).

Result (19a) follows from Lemma \ref{T21}(e). By \eqref{g3},
\begin{align}
\{M^{(1,3)}\} \supseteq \{C^{-1}B^{(1,3)}A^{-1}\} & \Leftrightarrow M^{*}MC^{-1}B^{(1,3)}A^{-1} = M^{*} \ \mbox{for all $B^{(1,3)}$} \nb
\\
& \Leftrightarrow \R(A^{*}AB) = \R(B)  \ \mbox{(by Lemma \ref{T21}(g))}.
\label{235}
\end{align}
 Also by \eqref{210},
\begin{align}
& \{M^{(1,3)}\} \subseteq \{C^{-1}B^{(1,3)}A^{-1}\}  \Leftrightarrow \{CM^{(1,3)}A\} \subseteq \{B^{(1,3)}\} \nb
\\
& \Leftrightarrow B^{*}BCM^{(1,3)}A = B^{*}  \ \mbox{for all $M^{(1,3)}$} \nb
\\
& \Leftrightarrow \R(A^{*}AB) = \R(B)  \ \mbox{(by Lemma \ref{T21}(i))},
\label{236}
\end{align}
Combining \eqref{235} and \eqref{236} leads to Result (19b).

Result (20a) follows from Lemma \ref{T21}(e). By \eqref{g3},
\begin{align}
\{M^{(1,3)}\} \supseteq \{C^{-1}B^{(1,4)}A^{-1}\} & \Leftrightarrow M^{*}MC^{-1}B^{(1,4)}A^{-1} = M^{*} \
\mbox{for all $B^{(1,4)}$} \nb
\\
& \Leftrightarrow  B =0 \ {\rm or} \ r(B) = m \ \mbox{(by Lemma \ref{T21}(f))},
\label{237}
\end{align}
thus establishing Result (20b).  By \eqref{210},
\begin{align}
& \{M^{(1,3)}\} \subseteq \{C^{-1}B^{(1,4)}A^{-1}\}  \Leftrightarrow \{CM^{(1,3)}A\} \subseteq \{B^{(1,4)}\} \nb
\\
& \Leftrightarrow CM^{(1,3)}ABB^{*} = B^{*}  \ \mbox{for all $M^{(1,3)}$} \nb
\\
& \Leftrightarrow B =0 \ {\rm or} \ r(B) = n \ \mbox{(by Lemma \ref{T21}(f))},
\label{238}
\end{align}
thus establishing Result (20c).

Combining \eqref{241} and \eqref{242} leads to Result (20c).  Combining Results (20b) and (20c) leads to Result (20d).

%
%
%

By \eqref{g3},
\begin{align}
 \{M^{(1,3)}\} \cap \{C^{-1}B^{(1,2,3)}A^{-1}\} \neq \emptyset & \Leftrightarrow M^{*}MC^{-1}B^{(1,2,3)}A^{-1} = M^{*} \ \mbox{for a $B^{(1,2,3)}$} \nb
\\
& \Leftrightarrow \{M^{(1,3)}\} \supseteq \{C^{-1}B^{(1,2,3)}A^{-1}\} \nb
\\
& \Leftrightarrow M^{*}MC^{-1}B^{(1,2,3)}A^{-1} = M^{*} \
\mbox{for all $B^{(1,2,3)}$} \nb
\\
& \Leftrightarrow  \R(A^{*}AB) = \R(B) \ \mbox{(by Lemma \ref{T21}(g))}.
\label{239}
\end{align}
establishing Result (21a). By \eqref{210} and By \eqref{g5},
\begin{align}
& \{M^{(1,3)}\} \subseteq \{C^{-1}B^{(1,2,3)}A^{-1}\}  \Leftrightarrow \{CM^{(1,3)}A\} \subseteq \{B^{(1,2,3)}\} \nb
\\
& \Leftrightarrow B^{*}BCM^{(1,3)}A = B^{*}  \ {\rm and} \ r(CM^{(1,3)}A)  = r(B) \ \mbox{for all $M^{(1,3)}$} \nb
\\
& \Leftrightarrow \R(A^{*}AB) = \R(B) \ {\rm and} \ r(B) = \min\{m, \, n\}   \
\mbox{(by Lemma \ref{T21}(i), \eqref{135} and \eqref{136})},
\label{240}
\end{align}
establishing the equivalence the first and third terms in Result (21b). Combining this fact with Result (21a)
leads to the second equivalence in Result (21b).

Result (22a) follows from Lemma \ref{T21}(e). By \eqref{g3},
\begin{align}
\{M^{(1,3)}\} \supseteq \{C^{-1}B^{(1,2,4)}A^{-1}\} & \Leftrightarrow M^{*}MC^{-1}B^{(1,2,4)}A^{-1} = M^{*} \
\mbox{for all $B^{(1,2,4)}$} \nb
\\
& \Leftrightarrow  B =0 \ {\rm or} \ r(B) = m \ \mbox{(by Lemma \ref{T21}(f))},
\label{241}
\end{align}
thus establishing Result (22b). By \eqref{210},
\begin{align}
& \{M^{(1,3)}\} \subseteq \{C^{-1}B^{(1,2,4)}A^{-1}\}  \Leftrightarrow \{CM^{(1,3)}A\} \subseteq \{B^{(1,2,4)}\} \nb
\\
& \Leftrightarrow CM^{(1,3)}ABB^{*} = B^{*} \ {\rm and} \ r(CM^{(1,3)}A) = r(B) \ \mbox{for all $M^{(1,3)}$} \nb
\\
& \Leftrightarrow \{B =0 \ {\rm or} \ r(B) = n \} \ {\rm and} \ r(B) = \min\{m, \, n\} \ \mbox{(by Lemma \ref{T21}(f),
\eqref{135} and \eqref{136})} \nb
\\
& \Leftrightarrow  r(B) = n,
\label{242}
\end{align}
thus establishing Result (22c). Combining Results (22b) and (22c) leads to Result (22d).

By \eqref{g3},
\begin{align}
 \{M^{(1,3)}\} \cap \{C^{-1}B^{(1,3,4)}A^{-1}\} \neq \emptyset & \Leftrightarrow M^{*}MC^{-1}B^{(1,3,4)}A^{-1} = M^{*} \ \mbox{for a $B^{(1,3,4)}$} \nb
\\
& \{M^{(1,3)}\} \supseteq \{C^{-1}B^{(1,3,4)}A^{-1}\} \nb
\\
& \Leftrightarrow M^{*}MC^{-1}B^{(1,3,4)}A^{-1} = M^{*} \
\mbox{for all $B^{(1,3,4)}$} \nb
\\
& \Leftrightarrow  \R(A^{*}AB) = \R(B) \ \mbox{(by Lemma \ref{T21}(g))},
\label{243}
\end{align}
thus establishing Result (23a).

Also by \eqref{210},
\begin{align}
& \{M^{(1,3)}\} \subseteq \{C^{-1}B^{(1,3,4)}A^{-1}\}  \Leftrightarrow \{CM^{(1,3)}A\} \subseteq \{B^{(1,3,4)}\} \nb
\\
& \Leftrightarrow B^{*}BCM^{(1,3)}A = B^{*} \ {\rm and} \ CM^{(1,3)}ABB^{*} = B^{*}  \mbox{for all $M^{(1,3)}$} \  \mbox{(by \eqref{g7})} \nb
\\
& \Leftrightarrow \R(A^{*}AB) = \R(B) \  {\rm and} \ \{B =0 \ {\rm or} \ r(B) = n \}  \
\mbox{(by Lemma \ref{T21}(h) and (i))} \nb
\\
& \Leftrightarrow B =0 \ {\rm or} \ \{ r(B) = n \ {\rm and} \ \R(A^{*}AB) = \R(B)\}.
\label{244}
\end{align}
establishing the equivalence the first and third terms in Result (23b). Combining this fact with Result (23a)
leads to the second equivalence in Result (23b).

Result (24) follows from \eqref{g3} and Lemma \ref{T21}(g).

Results (25a)--(32) are obtained from Lemma \ref{T11},  Results (17a)--(24), and matrix replacements.

Result (33a) follows from \eqref{210} and \eqref{213}. By \eqref{g5},
\begin{align}
& \{M^{(1,2,3)}\} \supseteq \{C^{-1}B^{(1)}A^{-1}\} \nb
\\
& \Leftrightarrow M^{*}MC^{-1}B^{(1)}A^{-1} = M^{*}  \ {\rm and} \ r(C^{-1}B^{(1)}A^{-1}) = r(M) \ \mbox{for all $B^{(1)}$} \nb
\\
& \Leftrightarrow  \{B =0 \ {\rm or} \ r(B) = m\} \ {\rm and} \ r(B) = \min\{m, \, n\}  \ \mbox{(by Lemma \ref{T21}(f), \eqref{135} and \eqref{136})} \nb
\\
& \Leftrightarrow  r(B) = m,
\label{245}
\end{align}
thus establishing the equivalence the first and third terms in Result (33b). Combining this fact with Result (33a)
leads to the second equivalence in Result (33b).

Result (34a) follows from \eqref{210}, \eqref{213} and $r(M^{(1,2,3)}) = r(M) = r(B)$. By \eqref{g5},
\begin{align}
& \{M^{(1,2,3)}\} \supseteq \{C^{-1}B^{(1,2)}A^{-1}\} \nb
\\
& \Leftrightarrow M^{*}MC^{-1}B^{(1,2)}A^{-1} = M^{*}  \ {\rm and} \ r(C^{-1}B^{(1,2)}A^{-1}) = r(M) \
\mbox{for all $B^{(1,2)}$} \nb
\\
& \Leftrightarrow  B =0 \ {\rm or} \ r(B) = m  \ \mbox{(by Lemma \ref{T21}(f))},
\label{246}
\end{align}
thus establishing the equivalence the first and third terms in Result (34b). Combining this fact with Result (34a)
leads to the second equivalence in Result (34b).

By \eqref{g5},
\begin{align}
\{M^{(1,2,3)}\} \cap \{C^{-1}B^{(1,3)}A^{-1}\} \neq \emptyset & \Leftrightarrow M^{*}MC^{-1}B^{(1,3)}A^{-1} = M^{*} \ \mbox{for a $B^{(1,3)}$} \nb
\\
& \Leftrightarrow \R(A^{*}AB) = \R(B)  \
\mbox{(by Lemma \ref{T21}(g))},
\label{247}
\end{align}
and by \eqref{210},
\begin{align}
& \{M^{(1,2,3)}\} \subseteq \{C^{-1}B^{(1,3)}A^{-1}\}  \Leftrightarrow \{CM^{(1,2,3)}A\} \subseteq \{B^{(1,3)}\} \nb
\\
& \Leftrightarrow B^{*}BCM^{(1,2,3)}A = B^{*} \ \mbox{for all $M^{(1,2,3)}$} \nb
\\
& \Leftrightarrow \R(A^{*}AB) = \R(B)  \ \mbox{(by Lemma \ref{T21}(i))},
\label{248}
\end{align}
Combining \eqref{247} and \eqref{248} leads to Result (35a).
By \eqref{g5},
\begin{align}
& \{M^{(1,2,3)}\} \supseteq \{C^{-1}B^{(1,3)}A^{-1}\}  \nb
\\
& M^{*}MC^{-1}B^{(1,3)}A^{-1} = M^{*} \ {\rm and} \ r(C^{-1}B^{(1,3)}A^{-1}) = r(M) \
\mbox{for all $B^{(1,3)}$}
 \nb
\\
& \Leftrightarrow \R(A^{*}AB) = \R(B)  \  {\rm and} \ r(B) = \min\{m, \, n\}   \
\mbox{(by Lemma \ref{T21}(g), \eqref{135} and \eqref{136})}.
\label{248a}
\end{align}
thus establishing the equivalence the first and third terms in Result (35b).
Combining this fact with Result (35a) leads to the second equivalence in Result (35b).

Result (36a) follows from \eqref{210} and Lemma \ref{T21}(e). By \eqref{g5},
\begin{align}
& \{M^{(1,2,3)}\} \supseteq \{C^{-1}B^{(1,4)}A^{-1}\} \nb
\\
& \Leftrightarrow M^{*}MC^{-1}B^{(1,4)}A^{-1} = M^{*} \ {\rm and} \ r(C^{-1}B^{(1,4)}A^{-1}) = r(M) \mbox{for all $B^{(1,4)}$} \nb
\\
&\Leftrightarrow  M^{*}MC^{-1}B^{(1,4)}A^{-1} = M^{*} \ {\rm and} \ r(B^{(1,4)}) = r(B) \mbox{for all $B^{(1,4)}$} \ \
\mbox{(by Lemma \ref{T21}(g))} \nb
\\
& \Leftrightarrow \{B =0 \ {\rm or} \ r(B) = m \} \ {\rm and} \ r(B) = \min\{m, \, n\}  \ \mbox{(by Lemma \ref{T21}(f),
\eqref{135} and \eqref{136})} \nb
\\
& \Leftrightarrow  r(B) = m,
\label{249}
\end{align}
thus establishing Result (36b).  Also by \eqref{210},
\begin{align}
& \{M^{(1,2,3)}\} \subseteq \{C^{-1}B^{(1,4)}A^{-1}\}  \Leftrightarrow \{CM^{(1,2,3)}A\} \subseteq \{B^{(1,4)}\} \nb
\\
& \Leftrightarrow CM^{(1,2,3)}ABB^{*} = B^{*} \ \mbox{for all $M^{(1,2,3)}$} \nb
\\
& \Leftrightarrow  B =0 \ {\rm or} \ r(B) = n   \ \mbox{(by Lemma \ref{T21}(h))},
\label{250}
\end{align}
thus establishing Result (36c). Combining Results (36b) and (36c) leads to Result (36d).

By \eqref{210},
\begin{align}
& \{M^{(1,2,3)}\} \cap \{C^{-1}B^{(1,2,3)}A^{-1}\}  \neq \emptyset \nb
\\
& \Leftrightarrow M^{*}MC^{-1}B^{(1,2,3)}A^{-1} = M^{*} \ {\rm and} \ r(C^{-1}B^{(1,2,3)}A^{-1}) = r(M)  \ \mbox{for a $B^{(1,2,3)}$} \nb
\\
& \Leftrightarrow \{M^{(1,2,3)}\}  \supseteq \{C^{-1}B^{(1,2,3)}A^{-1}\} \nb
\\
&\Leftrightarrow M^{*}MC^{-1}B^{(1,2,3)}A^{-1} = M^{*} \ {\rm and} \ r(C^{-1}B^{(1,2,3)}A^{-1}) = r(M) \ \mbox{for all $B^{(1,2,3)}$} \nb
\\
&\Leftrightarrow M^{*}MC^{-1}B^{(1,2,3)}A^{-1} = M^{*} \nb
\\
& \Leftrightarrow \R(A^{*}AB) = \R(B)  \ \mbox{(by Lemma \ref{T21}(g))},
\label{251}
\end{align}
and by \eqref{210},
\begin{align}
& \{M^{(1,2,3)}\} \subseteq \{C^{-1}B^{(1,2,3)}A^{-1}\} \Leftrightarrow \{CM^{(1,2,3)}A\} \subseteq \{B^{(1,2,3)}\} \nb
\\
& \Leftrightarrow B^{*}BCM^{(1,2,3)}A = B^{*} \ \mbox{for all $M^{(1,2,3)}$} \nb
\\
& \Leftrightarrow \R(A^{*}AB) = \R(B)  \ \mbox{(by Lemma \ref{T21}(i))},
\label{252}
\end{align}
Combining \eqref{251} and \eqref{252} leads to Result (37).

Result (38a) follows from \eqref{210} and Lemma \ref{T21}(e). By \eqref{g5},
\begin{align}
 \{M^{(1,2,3)}\}  \supseteq \{C^{-1}B^{(1,2,4)}A^{-1}\} &\Leftrightarrow M^{*}MC^{-1}B^{(1,2,4)}A^{-1} = M^{*}  \ \mbox{for all $B^{(1,2,4)}$} \nb
\\
& \Leftrightarrow B =0 \ {\rm or} \ r(B) = m   \ \mbox{(by Lemma \ref{T21}(h))},
\label{253}
\end{align}
thus establishing Result (38b). By \eqref{210},
\begin{align}
& \{M^{(1,2,3)}\} \subseteq \{C^{-1}B^{(1,2,4)}A^{-1}\}  \Leftrightarrow \{CM^{(1,2,3)}A\} \subseteq \{B^{(1,2,4)}\} \nb
\\
& \Leftrightarrow CM^{(1,2,3)}ABB^{*} = B^{*} \ \mbox{for all $M^{(1,2,3)}$}   \ \mbox{(by \eqref{g6})} \nb
\\
& \Leftrightarrow B =0 \ {\rm or} \ r(B) = n   \ \mbox{(by Lemma \ref{T21}(h))},
\label{254}
\end{align}
thus establishing Result (38c). Combining Results (38b) and (38c) leads to Result (38d).

By \eqref{g5},
\begin{align}
& \{M^{(1,2,3)}\} \cap \{C^{-1}B^{(1,3,4)}A^{-1}\} \neq \emptyset \nb
\\
&\Leftrightarrow M^{*}MC^{-1}B^{(1,3,4)}A^{-1} = M^{*} \ {\rm and} \ r(C^{-1}B^{(1,3,4)}A^{-1}) = r(M) \mbox{for a $B^{(1,3,4)}$} \nb
\\
& \Leftrightarrow \R(A^{*}AB) = \R(B)  \ \mbox{(by Lemma \ref{T21}(g))},
\label{255}
\end{align}
thus establishing Result (39a). Also by \eqref{g5},
\begin{align}
& \{M^{(1,2,3)}\}  \supseteq \{C^{-1}B^{(1,3,4)}A^{-1}\} \nb
\\
&\Leftrightarrow M^{*}MC^{-1}B^{(1,3,4)}A^{-1} = M^{*}  \ {\rm and} \ r(C^{-1}B^{(1,3,4)}A^{-1}) = r(M) \ \mbox{for all $B^{(1,3,4)}$} \nb
\\
& \Leftrightarrow \R(A^{*}AB) = \R(B)  \  {\rm and} \ r(B) = \min\{m, \, n\} \ \mbox{(by Lemma \ref{T21}(g), \eqref{135} and \eqref{136})},
\label{256}
\end{align}
thus establishing Result (39b).
and by \eqref{210},
\begin{align}
& \{M^{(1,2,3)}\} \subseteq \{C^{-1}B^{(1,3,4)}A^{-1}\} \Leftrightarrow \{CM^{(1,2,3)}A\} \subseteq \{B^{(1,3,4)}\} \nb
\\
& \Leftrightarrow B^{*}BCM^{(1,2,3)}A = B^{*} \ {\rm and} \ CM^{(1,2,3)}ABB^{*} = B^{*} \  \mbox{for all $M^{(1,2,3)}$}
\ \mbox{(by \eqref{g7})} \nb
\\
& \Leftrightarrow \R(A^{*}AB) = \R(B) \ {\rm and} \ \{B =0 \ {\rm or} \ r(B) = n\}  \ \mbox{(by Lemma \ref{T21}(h) and (i))},
\label{257}
\end{align}
thus establishing Result (39c).

Result (40) follows from \eqref{g5} and Lemma \ref{T21}(g).

Results (41a)--(48) are obtained from Lemma \ref{T11}, Results (33a)--(40), and matrix replacements.

Result (49a) follows from \eqref{g1}, \eqref{210},  and \eqref{213}.  By \eqref{g7},
\begin{align}
& \{M^{(1,3,4)}\} \supseteq  \{C^{-1}B^{(1)}A^{-1}\} \nb
\\
& \Leftrightarrow M^{*}MC^{-1}B^{(1)}A^{-1} = M^{*}  \ {\rm and} \
C^{-1}B^{(1)}A^{-1}MM^{*} = M^{*}  \  \mbox{for all $B^{(1)}$} \nb
\\
& \Leftrightarrow \{B =0 \ {\rm or} \ r(B) = m\}  \ {\rm and} \ \{B =0 \ {\rm or} \ r(B) = n\}  \ \mbox{(by Lemma \ref{T21}(f))} \nb
\\
& \Leftrightarrow  B=0 \ {\rm or} \ r(B) = m  = n,
\label{258}
\end{align}
thus establishing the equivalence the first and third terms in Result (49b). Combining this fact with Result (49a)
leads to the second equivalence in Result (49b).

Result (50a) follows from \eqref{28} and \eqref{213}.  By \eqref{g7},
\begin{align}
& \{M^{(1,3,4)}\} \supseteq  \{C^{-1}B^{(1,2)}A^{-1}\} \nb
\\
& \Leftrightarrow M^{*}MC^{-1}B^{(1,2)}A^{-1} = M^{*}  \ {\rm and} \
C^{-1}B^{(1,2)}A^{-1}MM^{*} = M^{*}  \  \mbox{for all $B^{(1,2)}$} \nb
\\
& \Leftrightarrow \{B =0 \ {\rm or} \ r(B) = m\}  \ {\rm and} \ \{B =0 \ {\rm or} \ r(B) = n\}  \ \mbox{(by Lemma \ref{T21}(f))} \nb
\\
& \Leftrightarrow  B=0 \ {\rm or} \ r(B) = m  = n,
\label{259}
\end{align}
thus establishing Result (50b). By \eqref{210},
\begin{align}
& \{M^{(1,3,4)}\} \subseteq  \{C^{-1}B^{(1,2)}A^{-1}\}  \Leftrightarrow \{CM^{(1,3,4)}A\} \supseteq  \{B^{(1,2)}\} \nb
\\
& \Leftrightarrow BCM^{(1,3,4)}AB = B  \ {\rm and} \ r(CM^{(1,3,4)}AB) = r(B) \ \mbox{for all $M^{(1,3,4)}$} \nb
\\
& \Leftrightarrow  r(B) = \min\{m, \, n\} \ \mbox{(by \eqref{213}, \eqref{135} and \eqref{136})} \nb
\\
& \Leftrightarrow  r(B)= m \ {\rm or} \ r(B) = n,
\label{260}
\end{align}
thus establishing Result (50c). Combining Results (50b) and (50c) leads to Result (50d).

By \eqref{g7},
\begin{align}
& \{M^{(1,3,4)}\} \cap \{C^{-1}B^{(1,3)}A^{-1}\} \neq \emptyset \nb
\\
&\Leftrightarrow M^{*}MC^{-1}B^{(1,3)}A^{-1} = M^{*} \ {\rm and} \ C^{-1}B^{(1,3)}A^{-1}MM^{*} = M^{*} \mbox{for a $B^{(1,3)}$} \nb
\\
& \Leftrightarrow \R(A^{*}AB) = \R(B)  \ \mbox{(by Lemma \ref{T21}(e) and (g))},
\label{261}
\end{align}
and by \eqref{210},
\begin{align}
 \{M^{(1,3,4)}\} \subseteq  \{C^{-1}B^{(1,3)}A^{-1}\} & \Leftrightarrow \{CM^{(1,3,4)}A\} \subseteq  \{B^{(1,3)}\} \nb
\\
&\Leftrightarrow B^{*}BCM^{(1,3,4)}A = B^{*} \ \mbox{for all $M^{(1,3,4)}$} \nb
\\
& \Leftrightarrow \R(A^{*}AB) = \R(B)  \ \mbox{(by Lemma \ref{T21}(i))}.
\label{262}
\end{align}
Combining \eqref{261} and \eqref{262} leads to Result (51a). by \eqref{210},
\begin{align}
& \{M^{(1,3,4)}\} \supseteq  \{C^{-1}B^{(1,3)}A^{-1}\} \nb
\\
& \Leftrightarrow M^{*}MC^{-1}B^{(1,3)}A^{-1} = M^{*}  \ {\rm and} \
C^{-1}B^{(1,3)}A^{-1}MM^{*} = M^{*}  \  \mbox{for all $B^{(1,3)}$} \nb
\\
& \Leftrightarrow \R(A^{*}AB) = \R(B)  \ {\rm and} \ \{B =0 \ {\rm or} \ r(B) = n\}  \ \mbox{(by Lemma \ref{T21}(f) and (g))} \nb
\\
& \Leftrightarrow  B =0  \ {\rm or} \ \{ r(B) = n \ {\rm and} \ \R(A^{*}AB) = \R(B) \},
\label{263}
\end{align}
thus establishing Result (51b).

Results (52a) and (52b) are obtained from Lemma \ref{T11}(e), Results (51a) and (51), and matrix replacements.

By \eqref{g7},
\begin{align}
& \{M^{(1,3,4)}\} \cap \{C^{-1}B^{(1,2,3)}A^{-1}\} \neq \emptyset \nb
\\
&\Leftrightarrow M^{*}MC^{-1}B^{(1,2,3)}A^{-1} = M^{*} \ {\rm and} \ C^{-1}B^{(1,2,3)}A^{-1}MM^{*} = M^{*}
\mbox{for a $B^{(1,2,3)}$} \nb
\\
& \Leftrightarrow \R(A^{*}AB) = \R(B)  \ \mbox{(by Lemma \ref{T21}(e) and (g))},
\label{264}
\end{align}
thus establishing Result (53a).   By \eqref{210},
\begin{align}
& \{M^{(1,3,4)}\} \supseteq  \{C^{-1}B^{(1,2,3)}A^{-1}\} \nb
\\
& \Leftrightarrow M^{*}MC^{-1}B^{(1,2,3)}A^{-1} = M^{*}  \ {\rm and} \
C^{-1}B^{(1,2,3)}A^{-1}MM^{*} = M^{*}  \  \mbox{for all $B^{(1,2,3)}$} \nb
\\
& \Leftrightarrow \R(A^{*}AB) = \R(B)  \ {\rm and} \ \{B =0 \ {\rm or} \ r(B) = n\}  \ \mbox{(by Lemma \ref{T21}(f) and (g))} \nb
\\
& \Leftrightarrow  B =0  \ {\rm or} \ \{ r(B) = n \ {\rm and} \ \R(A^{*}AB) = \R(B) \},
\label{265}
\end{align}
thus establishing Result (53b). By \eqref{210},
\begin{align}
& \{M^{(1,3,4)}\} \subseteq  \{C^{-1}B^{(1,2,3)}A^{-1}\}  \Leftrightarrow \{CM^{(1,3,4)}A\} \subseteq  \{B^{(1,2,3)}\} \nb
\\
&\Leftrightarrow B^{*}BCM^{(1,3,4)}A = B^{*} \ {\rm and} \  r(CM^{(1,3,4)}A) = r(B) \ \mbox{for all $M^{(1,3,4)}$} \nb
\\
& \Leftrightarrow \R(A^{*}AB) = \R(B)  \  {\rm and} \ r(B) = \min\{m, \, n\} \  \mbox{(by Lemma \ref{T21}(g), \eqref{135} and \eqref{136})},
\label{266}
\end{align}
thus establishing Result (53c). Combining Results (53b) and (53c) leads to Result (53d).

Results (54a)--(54d) are obtained from Lemma \ref{T11}(e), Results (53a)--(53d), and matrix replacements.

By \eqref{g7},
\begin{align}
& \{M^{(1,3,4)}\} \cap \{C^{-1}B^{(1,3,4)}A^{-1}\} \neq \emptyset \nb
\\
&\Leftrightarrow M^{*}MC^{-1}B^{(1,3,4)}A^{-1} = M^{*} \ {\rm and} \ C^{-1}B^{(1,3,4)}A^{-1}MM^{*} = M^{*} \ \mbox{for a $B^{(1,3,4)}$} \nb
\\
& \{M^{(1,3,4)}\} \supseteq  \{C^{-1}B^{(1,3,4)}A^{-1}\} \nb
\\
& \Leftrightarrow M^{*}MC^{-1}B^{(1,3)}A^{-1} = M^{*}  \ {\rm and} \ C^{-1}B^{(1,3)}A^{-1}MM^{*} = M^{*}  \  \mbox{for all $B^{(1,3,4)}$} \nb
\\
& \Leftrightarrow \R(A^{*}AB) = \R(B)  \ {\rm and}  \ \R(CC^{*}B^{*}) = \R(B^{*}) \  \mbox{(by Lemma \ref{T21}(g))},
\label{267}
\end{align}
and by \eqref{210},
\begin{align}
& \{M^{(1,3,4)}\} \subseteq  \{C^{-1}B^{(1,3,4)}A^{-1}\} \Leftrightarrow \{CM^{(1,3,4)}A\} \subseteq  \{B^{(1,3,4)}\} \nb
\\
&\Leftrightarrow B^{*}BCM^{(1,3,4)}A = B^{*} \ {\rm and} \  CM^{(1,3,4)}ABB^{*} = B^{*} \ \mbox{for all $M^{(1,3,4)}$} \nb
\\
& \Leftrightarrow \R(A^{*}AB) = \R(B) \ {\rm and} \ \R(CC^{*}B^{*}) = \R(B^{*})  \ \mbox{(by Lemma \ref{T21}(i))}.
\label{268}
\end{align}
Combining \eqref{267} and \eqref{268} leads to Result (55).

By \eqref{210},
\begin{align}
 \{M^{(1,3,4)}\} \ni C^{-1}B^{\dag}A^{-1}&\Leftrightarrow
 M^{*}MC^{-1}B^{\dag}A^{-1} = M^{*}  \ {\rm and} \ C^{-1}B^{\dag}A^{-1}MM^{*} = M^{*}   \nb
\\
& \Leftrightarrow \R(A^{*}AB) = \R(B) \ {\rm and} \ \R(CC^{*}B^{*}) = \R(B^{*})  \ \mbox{(by Lemma \ref{T21}(g))},
\label{269}
\end{align}
establishing Result (56).

Finally, we leave the verifications of Results (57)--(64) to the reader.
\end{proof}

We have presented a classification analysis to \eqref{12} using the elementary matrix range and rank method
and established hundreds of necessary and sufficient conditions for \eqref{12} to hold for the eight commonly-used types of generalized inverses of matrices. With doubt, we can use the previous results to solve many calculation problems on generalized inverses of matrix products, for example, when both $A$ and $C$ are unitary matrices, that is, If $AA^{*} = A^{*}A = I_m$ and $CC^{*} = C^{*}C =I_n$, then Theorem \ref{T22} reduces to a family of trivial results.

If $A$, $B$, and $C$ happen to be square matrices of the same size, and $C = A^{-1}$, then \eqref{12} can be written as
\begin{align}
(ABA^{-1})^{(i,\ldots,j)} = AB^{(k,\ldots,l)}A^{-1},
\label{274}
\end{align}
which are covariance equalities for generalized inverses of matrices. The special case $(ABA^{-1})^{\dag} = AB^{\dag}A^{-1}$
was approached by several authors; see, e.g.,  \cite{Ali:2013,MC:1992,Rob:1984,Rob:1985,Sch:1983},
and the relevant literature quoted there.

\begin{corollary} \label{T23}
Let $A, \, B \in {\mathbb C}^{m \times m}$ and assume that $A$ is nonsingular$.$ Also denote $M = ABA^{-1}.$ Then
\begin{enumerate}
\item[{\rm (a)}] The following results hold
$$
\begin{array}{ll}
 \{M^{(1)}\} = \{AB^{(1)}A^{-1}\}, & \{M^{(1)}\} \supseteq \{AB^{(1,2)}A^{-1}\},
\\
 \{M^{(1)}\} \supseteq \{AB^{(1,3)}A^{-1}\}  &   \{M^{(1)}\} \supseteq \{AB^{(1,4)}A^{-1}\},
\\
 \{M^{(1)}\} \supseteq \{AB^{(1,2,3)}A^{-1}\},  &  \{M^{(1)}\} \supseteq \{AB^{(1,2,4)}A^{-1}\},
\\
 \{M^{(1)}\} \supseteq \{AB^{(1,3,4)}A^{-1}\},  &   \{M^{(1)}\} \ni AB^{\dag}A^{-1},
\\
 \{M^{(1,2)}\} \subseteq \{AB^{(1)}A^{-1}\},  &   \{M^{(1,2)}\} = \{AB^{(1,2)}A^{-1}\},
\\
\{M^{(1,2)}\} \cap \{AB^{(1,3)}A^{-1}\}\neq \emptyset, & \{M^{(1,2)}\} \cap \{AB^{(1,4)}A^{-1}\}\neq \emptyset,
\\
 \{M^{(1,2)}\} \supseteq \{AB^{(1,2,3)}A^{-1}\}, & \{M^{(1,2)}\} \supseteq \{AB^{(1,2,4)}A^{-1}\},
\\
 \{M^{(1,2)}\} \cap \{AB^{(1,3,4)}A^{-1}\}\neq \emptyset, & \{M^{(1,2)}\} \ni AB^{\dag}A^{-1},
\\
 \{M^{(1,3)}\} \subseteq \{AB^{(1)}A^{-1}\}, & \{M^{(1,3)}\} \cap \{AB^{(1,2)}A^{-1}\}\neq \emptyset,
\\
 \{M^{(1,3)}\} \cap \{AB^{(1,3)}A^{-1}\}\neq \emptyset, & \{M^{(1,3)}\} \cap \{AB^{(1,4)}A^{-1}\}\neq \emptyset,
\\
\{M^{(1,3)}\} \cap \{AB^{(1,2,4)}A^{-1}\}\neq \emptyset, &  \{M^{(1,4)}\} \subseteq \{AB^{(1)}A^{-1}\},
\\
 \{M^{(1,4)}\} \cap \{AB^{(1,2)}A^{-1}\}\neq \emptyset, & \{M^{(1,4)}\} \cap \{AB^{(1,3)}A^{-1}\}\neq \emptyset,
\\
 \{M^{(1,4)}\} \cap \{AB^{(1,4)}A^{-1}\} \neq \emptyset, &  \{M^{(1,4)}\} \cap \{AB^{(1,2,3)}A^{-1}\}\neq \emptyset,
\\
 \{M^{(1,2,3)}\} \subseteq \{AB^{(1)}A^{-1}\}, & \{M^{(1,2,3)}\} \subseteq \{AB^{(1,2)}A^{-1}\},
\\
 \{M^{(1,2,3)}\} \cap \{AB^{(1,4)}A^{-1}\} \neq \emptyset, &  \{M^{(1,2,3)}\} \cap \{AB^{(1,2,4)}A^{-1}\} \neq \emptyset,
\\
 \{M^{(1,2,4)}\} \subseteq \{AB^{(1)}A^{-1}\}, &  \{M^{(1,2,4)}\} \subseteq \{AB^{(1,2)}A^{-1}\},
\\
 \{M^{(1,2,4)}\} \cap \{AB^{(1,3)}A^{-1}\} \neq \emptyset, &  \{M^{(1,2,4)}\} \cap \{AB^{(1,2,3)}A^{-1}\} \neq \emptyset,
\\
 \{M^{(1,3,4)}\} \subseteq \{AB^{(1)}A^{-1}\}, &  \{M^{(1,3,4)}\} \cap \{AB^{(1,2)}A^{-1}\}\neq \emptyset,
\\
 M^{\dag} \in \{AB^{(1)}A^{-1}\},  & M^{\dag} \in \{AB^{(1,2)}A^{-1}\}.
\end{array}
$$

\item[{\rm (b)}] The following equivalent statements hold
\begin{align*}
& \{M^{(1)}\} \subseteq \{AB^{(1,2)}A^{-1}\} \Leftrightarrow  \{M^{(1)}\} = \{AB^{(1,2)}A^{-1}\}  \Leftrightarrow  \{M^{(1)}\} \subseteq \{AB^{(1,2,3)}A^{-1}\}
\\
& \Leftrightarrow \{M^{(1)}\} = \{AB^{(1,2,3)}A^{-1}\} \Leftrightarrow  \{M^{(1)}\} \subseteq \{AB^{(1,2,4)}A^{-1}\} \Leftrightarrow  \{M^{(1)}\} = \{AB^{(1,2,4)}A^{-1}\}
\\
& \Leftrightarrow \{M^{(1,2)}\} \supseteq \{AB^{(1)}A^{-1}\} \Leftrightarrow \{M^{(1,2)}\} = \{AB^{(1)}A^{-1}\} \Leftrightarrow \{M^{(1,2)}\} \supseteq \{AB^{(1,3)}A^{-1}\}
\\
& \Leftrightarrow \{M^{(1,2)}\}  = \{AB^{(1,3)}A^{-1}\}\Leftrightarrow \{M^{(1,2)}\} \supseteq \{AB^{(1,4)}A^{-1}\} \Leftrightarrow \{M^{(1,2)}\}  = \{AB^{(1,4)}A^{-1}\}
\\
&\Leftrightarrow \{M^{(1,2)}\} \supseteq \{AB^{(1,3,4)}A^{-1}\} \Leftrightarrow \{M^{(1,2)}\} = \{AB^{(1,3,4)}A^{-1}\} \Leftrightarrow  \{M^{(1,3)}\} = \{AB^{(1,2)}A^{-1}\}
\\
& \Leftrightarrow \{M^{(1,3)}\} \subseteq \{AB^{(1,2,4)}A^{-1}\} \Leftrightarrow \{M^{(1,3)}\} = \{AB^{(1,2,4)}A^{-1}\} \Leftrightarrow \{M^{(1,4)}\} \subseteq \{AB^{(1,2)}A^{-1}\}
\\
& \Leftrightarrow \{M^{(1,4)}\} = \{AB^{(1,2)}A^{-1}\} \Leftrightarrow \{M^{(1,4)}\} \subseteq \{AB^{(1,2,3)}A^{-1}\} \Leftrightarrow  \{M^{(1,4)}\} = \{AB^{(1,2,3)}A^{-1}\}
\\
& \Leftrightarrow \{M^{(1,2,3)}\} \supseteq  \{AB^{(1)}A^{-1}\} \Leftrightarrow \{M^{(1,2,3)}\} = \{AB^{(1)}A^{-1}\} \Leftrightarrow\{M^{(1,2,3)}\} \supseteq  \{AB^{(1,4)}A^{-1}\}
\\
& \Leftrightarrow \{M^{(1,2,3)}\} = \{AB^{(1,4)}A^{-1}\} \Leftrightarrow\{M^{(1,2,3)}\} \supseteq  \{AB^{(1,3,4)}A^{-1}\} \Leftrightarrow  \{M^{(1,2,3)}\} \subseteq  \{AB^{(1,3,4)}A^{-1}\}
\\
& \Leftrightarrow  \{M^{(1,2,4)}\} \supseteq  \{AB^{(1)}A^{-1}\} \Leftrightarrow  \{M^{(1,2,4)}\} = \{AB^{(1)}A^{-1}\}  \Leftrightarrow \{M^{(1,2,4)}\} \supseteq  \{AB^{(1,3)}A^{-1}\}
\\
&\Leftrightarrow\{M^{(1,2,4)}\} = \{AB^{(1,3)}A^{-1}\} \Leftrightarrow \{M^{(1,2,4)}\} \supseteq  \{AB^{(1,3,4)}A^{-1}\} \Leftrightarrow \{M^{(1,2,4)}\} = \{AB^{(1,3,4)}A^{-1}\}
\\
&\Leftrightarrow  \{M^{(1,3,4)}\} \subseteq \{AB^{(1,2)}A^{-1}\} \Leftrightarrow \{M^{(1,3,4)}\} = \{AB^{(1,2)}A^{-1}\} \Leftrightarrow \{M^{(1,3,4)}\} \subseteq \{AB^{(1,2,3)}A^{-1}\}
\\
&\Leftrightarrow \{M^{(1,3,4)}\} = \{AB^{(1,2,3)}A^{-1}\} \Leftrightarrow \{M^{(1,3,4)}\} \subseteq \{AB^{(1,2,4)}A^{-1}\}  \Leftrightarrow \{M^{(1,3,4)}\} = \{AB^{(1,2,4)}A^{-1}\}
\\
& \Leftrightarrow r(B) = m.
\end{align*}

\item[{\rm (c)}] The following equivalent statements hold
\begin{align*}
& \{M^{(1)}\} \subseteq \{AB^{(1,3)}A^{-1}\} \Leftrightarrow \{M^{(1)}\} =\{AB^{(1,3)}A^{-1}\} \Leftrightarrow \{M^{(1)}\} \subseteq \{AB^{(1,4)}A^{-1}\}
\\
&  \Leftrightarrow
\{M^{(1)}\} = \{AB^{(1,4)}A^{-1}\} \Leftrightarrow  \{M^{(1)}\} \subseteq \{AB^{(1,3,4)}A^{-1}\} \Leftrightarrow
\{M^{(1)}\} = \{AB^{(1,3,4)}A^{-1}\}
\\
& \Leftrightarrow \{M^{(1,2)}\} \subseteq \{AB^{(1,3)}A^{-1}\} \Leftrightarrow
 \{M^{(1,2)}\} \subseteq \{AB^{(1,4)}A^{-1}\} \Leftrightarrow \{M^{(1,2)}\} \subseteq \{AB^{(1,2,3)}A^{-1}\}
 \\
& \Leftrightarrow \{M^{(1,2)}\} = \{AB^{(1,2,3)}A^{-1}\} \Leftrightarrow \{M^{(1,2)}\} \subseteq \{AB^{(1,2,4)}A^{-1}\} \Leftrightarrow \{M^{(1,2)}\} = \{AB^{(1,2,4)}A^{-1}\}
\\
& \Leftrightarrow \{M^{(1,2)}\} \subseteq \{AB^{(1,3,4)}A^{-1}\} \Leftrightarrow  \{M^{(1,3)}\} \supseteq \{AB^{(1)}A^{-1}\} \Leftrightarrow
\{M^{(1,3)}\} = \{AB^{(1)}A^{-1}\}
\\
&  \Leftrightarrow \{M^{(1,3)}\} \supseteq \{AB^{(1,2)}A^{-1}\} \Leftrightarrow \{M^{(1,3)}\} \supseteq \{AB^{(1,4)}A^{-1}\} \Leftrightarrow \{M^{(1,3)}\} \subseteq \{AB^{(1,4)}A^{-1}\}
\\
&\Leftrightarrow \{M^{(1,3)}\} = \{AB^{(1,4)}A^{-1}\} \Leftrightarrow
\{M^{(1,3)}\} \supseteq \{AB^{(1,2,4)}A^{-1}\} \Leftrightarrow \{M^{(1,4)}\} \supseteq \{AB^{(1)}A^{-1}\}
\\
& \Leftrightarrow \{M^{(1,3)}\} = \{AB^{(1)}A^{-1}\} \Leftrightarrow  \{M^{(1,4)}\} \supseteq \{AB^{(1,2)}A^{-1}\} \Leftrightarrow \{M^{(1,4)}\} \supseteq \{AB^{(1,3)}A^{-1}\}
\\
& \Leftrightarrow \{M^{(1,4)}\} \subseteq \{AB^{(1,3)}A^{-1}\} \Leftrightarrow  \{M^{(1,4)}\} = \{AB^{(1,3)}A^{-1}\}
\Leftrightarrow  \{M^{(1,4)}\} \supseteq \{AB^{(1,2,3)}A^{-1}\}
\\
& \Leftrightarrow  \{M^{(1,2,3)}\} \supseteq  \{AB^{(1,2)}A^{-1}\} \Leftrightarrow \{M^{(1,2,3)}\} = \{AB^{(1,2)}A^{-1}\}
\Leftrightarrow \{M^{(1,2,3)}\} \subseteq \{AB^{(1,4)}A^{-1}\}
 \\
 & \Leftrightarrow \{M^{(1,2,3)}\} \supseteq  \{AB^{(1,2,4)}A^{-1}\} \Leftrightarrow \{M^{(1,2,3)}\} \subseteq \{AB^{(1,2,4)}A^{-1}\} \Leftrightarrow \{M^{(1,2,3)}\} = \{AB^{(1,2,4)}A^{-1}\}
 \\
& \Leftrightarrow \{M^{(1,2,3)}\} \subseteq  \{AB^{(1,3,4)}A^{-1}\} \Leftrightarrow \{M^{(1,2,4)}\} \supseteq  \{AB^{(1,2)}A^{-1}\} \Leftrightarrow \{M^{(1,2,4)}\} = \{AB^{(1,2)}A^{-1}\}
\\
&  \Leftrightarrow   \{M^{(1,2,4)}\} \subseteq \{AB^{(1,3)}A^{-1}\} \Leftrightarrow \{M^{(1,2,4)}\} \supseteq  \{AB^{(1,2,3)}A^{-1}\}  \Leftrightarrow \{M^{(1,2,4)}\} \subseteq \{AB^{(1,2,3)}A^{-1}\}
\\
& \Leftrightarrow \{M^{(1,2,4)}\} = \{AB^{(1,2,3)}A^{-1}\} \Leftrightarrow \{M^{(1,2,4)}\} \subseteq \{AB^{(1,3,4)}A^{-1}\} \Leftrightarrow \{M^{(1,3,4)}\} \supseteq \{AB^{(1)}A^{-1}\}
\\
& \Leftrightarrow \{M^{(1,3,4)}\} = \{AB^{(1)}A^{-1}\} \Leftrightarrow \{M^{(1,3,4)}\} \supseteq \{AB^{(1,2)}A^{-1}\} \Leftrightarrow \{M^{(1,3,4)}\} \supseteq \{AB^{(1,3)}A^{-1}\}
\\
& \Leftrightarrow \{M^{(1,3,4)}\} = \{AB^{(1,3)}A^{-1}\} \Leftrightarrow \{M^{(1,3,4)}\} \supseteq \{AB^{(1,4)}A^{-1}\} \Leftrightarrow \{M^{(1,3,4)}\} = \{AB^{(1,4)}A^{-1}\}
\\
& \Leftrightarrow \{M^{(1,3,4)}\} \supseteq \{AB^{(1,2,3)}A^{-1}\} \Leftrightarrow  \{M^{(1,3,4)}\} \supseteq \{AB^{(1,2,4)}A^{-1}\} \Leftrightarrow B =0 \  or \ r(B) = m.
\end{align*}

\item[{\rm (d)}] The following equivalent statements hold
\begin{align*}
& \{M^{(1,3)}\} \supseteq  \{AB^{(1,3)}A^{-1}\}  \Leftrightarrow  \{M^{(1,3)}\} \subseteq  \{AB^{(1,3)}A^{-1}\}  \Leftrightarrow  \{M^{(1,3)}\} = \{AB^{(1,3)}A^{-1}\}
\\
& \Leftrightarrow \{M^{(1,3)}\} \cap \{AB^{(1,2,3)}A^{-1}\} \neq \emptyset \Leftrightarrow \{M^{(1,3)}\} = \{AB^{(1,2,3)}A^{-1}\} \Leftrightarrow   \{M^{(1,3)}\} \cap \{AB^{(1,3,4)}A^{-1}\}\neq \emptyset
\\
 & \Leftrightarrow \{M^{(1,3)}\} = \{AB^{(1,3,4)}A^{-1}\} \Leftrightarrow \{M^{(1,3)}\} \ni AB^{\dag}A^{-1} \Leftrightarrow
\{M^{(1,2,3)}\} \cap \{AB^{(1,3)}A^{-1}\} \neq \emptyset
\\
& \Leftrightarrow \{M^{(1,2,3)}\} = \{AB^{(1,3)}A^{-1}\}  \Leftrightarrow  \{M^{(1,2,3)}\} \cap \{AB^{(1,2,3)}A^{-1}\} \neq \emptyset  \Leftrightarrow
\{M^{(1,2,3)}\} = \{AB^{(1,2,3)}A^{-1}\}
\\
& \Leftrightarrow \{M^{(1,2,3)}\} \cap \{AB^{(1,3,4)}A^{-1}\} \neq \emptyset  \Leftrightarrow \{M^{(1,2,3)}\} \ni AB^{\dag}A^{-1} \Leftrightarrow  \{M^{(1,3,4)}\} \cap \{AB^{(1,3)}A^{-1}\}\neq \emptyset
\\
& \Leftrightarrow
\{M^{(1,3,4)}\} \subseteq \{AB^{(1,3)}A^{-1}\} \Leftrightarrow \{M^{(1,3,4)}\} \cap \{AB^{(1,2,3)}A^{-1}\}\neq \emptyset \Leftrightarrow M^{\dag} \in \{AB^{(1,3)}A^{-1}\}
\\
& \Leftrightarrow  M^{\dag} \in \{AB^{(1,2,3)}A^{-1}\} \Leftrightarrow \R(A^{*}AB) = \R(B).
\end{align*}

\item[{\rm (e)}] The following equivalent statements hold
\begin{align*}
& \{M^{(1,4)}\} \supseteq \{AB^{(1,4)}A^{-1}\} \Leftrightarrow
\{M^{(1,4)}\} \subseteq  \{AB^{(1,4)}A^{-1}\} \Leftrightarrow\{M^{(1,4)}\} = \{AB^{(1,4)}A^{-1}\}
\\
& \Leftrightarrow \{M^{(1,4)}\} \cap \{AB^{(1,2,4)}A^{-1}\} \neq \emptyset  \Leftrightarrow \{M^{(1,4)}\} = \{AB^{(1,2,4)}A^{-1}\} \Leftrightarrow \{M^{(1,4)}\} \cap \{AB^{(1,3,4)}A^{-1}\}\neq \emptyset
\\
& \Leftrightarrow \{M^{(1,4)}\} = \{AB^{(1,3,4)}A^{-1}\} \Leftrightarrow  \{M^{(1,4)}\} \ni AB^{\dag}A^{-1}
\Leftrightarrow \{M^{(1,2,4)}\} \cap \{AB^{(1,4)}A^{-1}\} \neq \emptyset
\\
& \Leftrightarrow\{M^{(1,2,4)}\} = \{AB^{(1,4)}A^{-1}\}  \Leftrightarrow \{M^{(1,2,4)}\} \cap \{AB^{(1,2,4)}A^{-1}\} \neq \emptyset  \Leftrightarrow \{M^{(1,2,4)}\} = \{AB^{(1,2,4)}A^{-1}\}
\\
& \Leftrightarrow   \{M^{(1,2,4)}\} \cap \{AB^{(1,3,4)}A^{-1}\} \neq \emptyset \Leftrightarrow\{M^{(1,2,4)}\} \ni AB^{\dag}A^{-1} \Leftrightarrow  \{M^{(1,3,4)}\} \cap \{AB^{(1,4)}A^{-1}\}\neq \emptyset
\end{align*}
\begin{align*}
& \Leftrightarrow \{M^{(1,3,4)}\} \subseteq \{AB^{(1,4)}A^{-1}\} \Leftrightarrow  \{M^{(1,3,4)}\} \cap \{AB^{(1,2,4)}A^{-1}\}\neq \emptyset \Leftrightarrow M^{\dag} \in \{AB^{(1,4)}A^{-1}\}
\\
&\Leftrightarrow   M^{\dag} \in \{AB^{(1,2,4)}A^{-1}\} \Leftrightarrow \R(A^{*}AB^{*}) = \R(B^{*}).
\end{align*}

\item[{\rm (f)}] The following equivalent statements hold
\begin{align*}
& \{M^{(1,3,4)}\} \cap \{AB^{(1,3,4)}A^{-1}\} \neq \emptyset  \Leftrightarrow \{M^{(1,3,4)}\} = \{AB^{(1,3,4)}A^{-1}\}  \Leftrightarrow  \{M^{(1,3,4)}\} \ni AB^{\dag}A^{-1}
\\
& \Leftrightarrow M^{\dag} \in \{AB^{(1,3,4)}A^{-1}\} \Leftrightarrow M^{\dag} = AB^{\dag}A^{-1} \Leftrightarrow \R(A^{*}AB) = \R(B) \ and \ \R(A^{*}AB^{*}) = \R(B^{*}). 
\end{align*}

\item[{\rm (g)}]  Under $A^{*}A = I_m,$ then the following equalities hold
\begin{align*}
& \{(ABA^{*})^{(1)}\} = \{AB^{(1)}A^{*}\}, \ \  \{(ABA^{*})^{(1,2)}\} = \{AB^{(1,2)}A^{*}\}, \ \ \{(ABA^{*})^{(1,3)}\} = \{AB^{(1,3)}A^{*}\},
\\
& \{(ABA^{*})^{(1,4)}\} = \{AB^{(1,4)}A^{*}\}, \ \ \{(ABA^{*})^{(1,2,3)}\} = \{AB^{(1,2,3)}A^{*}\}, \ \ \{M^{(1,2,4)}\} = \{AB^{(1,2,4)}A^{*}\},
\\
& \{(ABA^{*})^{(1,3,4)}\} = \{AB^{(1,3,4)}A^{*}\}, \ \ (ABA^{*})^{\dag} = AB^{\dag}A^{*}.
\end{align*}

\end{enumerate}
\end{corollary}

Finally, we give two consequences. Let $A, \, B \in {\mathbb C}^{m \times n}$.  Then it is easy to verify
\begin{align}
A + B = \frac{1}{2}[\,I_m, \, I_m\,]\begin{bmatrix} A  & B  \\ B  & A \end{bmatrix} \begin{bmatrix} I_n  \\ I_n \end{bmatrix} = \frac{1}{2}PNQ, \ \ P^{\dag} = \frac{1}{\sqrt{2}}\begin{bmatrix} I_m  \\ I_m \end{bmatrix},
 \ \ Q^{\dag} = \frac{1}{\sqrt{2}}[\,I_n, \, I_n\,].
\label{276a}
\end{align}

\begin{corollary} \label{T25}
Let $A, \, B \in {\mathbb C}^{m \times n}$ and $N$ be as given in \eqref{276a}$.$ Then the following seven set equalities and a matrix equality hold
\begin{align*}
& \{(A + B)^{(1)}\} = \left\{\frac{1}{2}[\,I_n, \, I_n\,]N^{(1)}\begin{bmatrix} I_m  \\ I_m \end{bmatrix} \right\}\!,  \ \ \ \ \ \ \ \ \ \ \ \, \{(A + B)^{(1,2)}\} =
\left\{\frac{1}{2}[\,I_n, \, I_n\,]N^{(1,2)}\begin{bmatrix} I_m  \\ I_m \end{bmatrix} \right\}\!,
\\
& \{(A + B)^{(1,3)}\} = \left\{\frac{1}{2}[\,I_n, \, I_n\,]N^{(1,3)}\begin{bmatrix} I_m  \\ I_m \end{bmatrix} \right\}\!, \ \ \ \ \ \  \ \ \,  \{(A + B)^{(1,4)}\} = \left\{\frac{1}{2}[\,I_n, \, I_n\,]N^{(1,4)}\begin{bmatrix} I_m  \\ I_m \end{bmatrix} \right\}\!,
\\
& \{(A + B)^{(1,2,3)}\} =
\left\{\frac{1}{2}[\,I_n, \, I_n\,]N^{(1,2,3)}\begin{bmatrix} I_m  \\ I_m \end{bmatrix} \right\}\!, \ \ \ \ \ \{(A + B)^{(1,2,4)}\} = \left\{\frac{1}{2}[\,I_n, \, I_n\,]N^{(1,2,4)}\begin{bmatrix} I_m  \\ I_m \end{bmatrix} \right\}\!,
\\
& \{(A + B)^{(1,3,4)}\} = \left\{\frac{1}{2}[\,I_n, \, I_n\,]N^{(1,3,4)}\begin{bmatrix} I_m  \\ I_m \end{bmatrix} \right\}\!, \ \ \ \ \ (A + B)^{\dag} = \frac{1}{2}[\,I_n, \, I_n\,]N^{\dag}\begin{bmatrix} I_m  \\ I_m \end{bmatrix}.
\end{align*}
\end{corollary}

Another pair of examples are given below. It is easy to verify that following two identities
\begin{align}
(\, I_m +  \alpha A +  \beta B \,) & = (\, I_m +  \alpha A \,)
[\, I_m -(\alpha\beta)(1 + \alpha)^{-1}(1 + \beta)^{-1}AB\,](\, I_m + \beta B \,),
\label{279}
\\
(\, I_m +  \alpha A +  \beta B \,) & = (\, I_m + \beta B\,)[\, I_m -(\alpha\beta)(1 + \alpha)^{-1}(1 + \beta)^{-1}BA\,](\, I_m +  \alpha A \,)
\label{280}
\end{align}
hold for two idempotents $A$ and $B$, where $\alpha \neq -1, 0$ and $\beta \neq -1, 0$; see \cite{Vet}. In this case, $I_m +  \alpha A$ and $I_m + \beta B$ are nonsingular, and thus two families of ROLs for generalized inverses
associated with \eqref{279} and \eqref{280} are given by
\begin{align*}
& (\, I_m - \lambda AB\,)^{(i,\ldots,j)} = (\, I_m + \beta B \,)(\, I_m +  \alpha A  + \beta B \,)^{(k,\ldots,l)} (\, I_m  + \alpha A \,),
\\
& (\, I_m  - \lambda BA\,)^{(i,\ldots,j)} = (\, I_m  + \alpha A \,)(\, I_m +  \alpha A  + \beta B \,)^{(k,\ldots,l)}(\,  I_m + \beta B\,),
\end{align*}
where $\lambda = (\alpha\beta)(1 + \alpha)^{-1}(1 + \beta)^{-1}$. Applying Theorem \ref{T22} to \eqref{279}, we obtain the following consequence.

\begin{corollary} \label{T26}
Let $A, \, B \in {\mathbb C}^{m \times m}$ be two idempotent matrices$,$ assume $\alpha \neq -1, 0$ and $\beta \neq -1, 0,$ and denote $\lambda = \alpha\beta(1 - \alpha)^{-1}(1 - \beta)^{-1}.$
Then the following matrix set equality
\begin{align*}
\{(\, I_m - \lambda AB\,)^{(1)} \}= \{(\, I_m + \beta B \,)(\, I_m +  \alpha A  + \beta B \,)^{(1)} (\, I_m  + \alpha A \,) \}
\end{align*}
holds$.$ In particular$,$ the ROL
\begin{align*}
(\, I_m - \lambda AB\,)^{\dag} = (\, I_m + \beta B \,)(\, I_m +  \alpha A  + \beta B \,)^{\dag} (\, I_m  + \alpha A \,)
\end{align*}
holds if and only if $\R[(\, I_m + \beta B \,)(\, I_m + \beta B \,)^{*}(\, I_m +  \alpha A +  \beta B \,)] = \R(I_m +  \alpha A +  \beta B)$ and $\R[(\, I_m  + \alpha A \,)^{*}(\, I_m  + \alpha A \,)(\, I_m +  \alpha A +  \beta B \,)^{*}] = \R[(\, I_m +  \alpha A +  \beta B \,)^{*}].$
\end{corollary}

%
%

\section[5]{Conclusions}

We have characterized a family of reverse order laws for a specified triple matrix product with applications using definitions and matrix rank formulas,
and have featured several examples that involve generalized inverses of matrices. We believe all these formulas and facts can be used in the
computations of various matrix expressions that involve products of matrices and their generalized inverses of matrices.

Note moreover that generalized inverses of elements can symbolically be defined in many other algebraic structures. Thus it would be of interest to consider the extension of the preceding results and facts to reverse order laws for generalized inverses of elements in other algebraic structures.

\end{document}